\documentclass[a4paper, 12pt]{amsart}

\usepackage{amssymb}
\usepackage{amsthm}
\usepackage{amsmath}
\usepackage[mathscr]{eucal}
\usepackage{comment}
\usepackage{graphicx}
\usepackage{enumitem}

\theoremstyle{plain}
\newtheorem{thm}{Theorem}[section]
\newtheorem{prop}[thm]{Proposition}
\newtheorem{cor}[thm]{Corollary}
\newtheorem{lem}[thm]{Lemma}

\newtheorem{ques}[thm]{Question}

\theoremstyle{definition}
\newtheorem{df}{Definition}[section]

\theoremstyle{remark}
\newtheorem{rmk}{Remark}[section]
\newtheorem*{ac}{Acknowledgements}

\newcommand{\zz}{\mathbb{Z}}

\newcommand{\rr}{\mathbb{R}}

\DeclareMathOperator{\card}{Card}

\DeclareMathOperator{\met}{Met}

\DeclareMathOperator{\ult}{UMet}

\newcommand{\grsp}{\mathscr{M}}
\DeclareMathOperator{\grdis}{\mathcal{GH}}

\DeclareMathOperator{\hdis}{\mathcal{HD}}

\DeclareMathOperator{\dis}{dis}

\newcommand{\qcube}{\mathbf{Q}}

\DeclareMathOperator{\lbdim}{\underline{\dim}_{B}}
\DeclareMathOperator{\ubdim}{\overline{\dim}_{B}}
\DeclareMathOperator{\hdim}{\dim_{H}}
\DeclareMathOperator{\tdim}{\dim_{T}}
\DeclareMathOperator{\pdim}{\dim_{P}}
\DeclareMathOperator{\adim}{\dim_{A}}
\DeclareMathOperator{\ulocdim}{\overline{\dim}_{\mathrm{loc}}}
\DeclareMathOperator{\llocdim}{\underline{\dim}_{\mathrm{loc}}}

\DeclareMathOperator{\ugrdis}{\mathcal{NA}}
\newcommand{\ultsp}{\mathscr{U}}

\newcommand{\mnum}{\mathrm{N}}

\newcommand{\dimset}{\mathscr{D}}
\newcommand{\linset}{\mathcal{L}}
\newcommand{\linsettwo}{\mathcal{R}}
\newcommand{\numset}{\mathbb{M}}

\newcommand{\shrinkset}{\mathbb{SH}}

\DeclareMathOperator{\pmet}{PMet}
\DeclareMathOperator{\pumet}{PUMet}
\DeclareMathOperator{\ooo}{\Omega}

\newcommand{\deltasecond}{H^{\times}}

\newcommand{\yoavoidmap}{W}

\newcommand{\pppp}{%
  \mathop{\ooalign{{\scalebox{1.5}{$\boxplus$}}\cr\hidewidth\raisebox{.45ex}{\scalebox{.3}{$\mkern30mu \boldsymbol{p}$}}\hidewidth}}}%
\makeatletter
\@addtoreset{equation}{section}

\makeatother

\begin{document}

\title[Topological embeddings]
{
Fractal dimensions in the 
Gromov--Hausdorff space
 }

\author[Yoshito Ishiki]
{Yoshito Ishiki}
\address[Yoshito Ishiki]
{\endgraf
Photonics Control Technology Team
\endgraf
RIKEN Center for Advanced Photonics
\endgraf
2-1 Hirasawa, Wako, Saitama 351-0198, Japan}
\email{yoshito.ishiki@riken.jp}


\maketitle

\begin{abstract}
In this paper,  
we first show that 
for all four  non-negative real numbers, 
there exists a Cantor ultrametric space
whose Hausdorff dimension, 
packing dimension, 
upper box dimension, 
and Assouad dimension 
are equal to given four numbers, respectively. 
Next, 
by constructing topological embeddings of 
an arbitrary  compact metrizable space
into the Gromov--Hausdorff space
using a direct sum of metrics spaces, 
we prove that 
the set of all compact metric spaces 
possessing 
prescribed  
topological dimension,  
and  four dimensions explained above, 
and the set of all compact  ultrametric spaces
are  path-connected and
 have  infinite topological dimension. 
This observation on ultrametrics  provides another proof of 
Qiu's theorem 
stating that 
 the ratio of the Archimedean and  non-Archimedean Gromov--Hausdorff distances  
is unbounded. 
\end{abstract}

\section{Introduction}\label{sec:intro}

In the present  paper,  
we mainly  deal with 
the topological dimension $\tdim X$, 
the Hausdorff dimension $\hdim(X, d)$, 
the packing dimension $\pdim(X, d)$, 
the upper box dimension $\ubdim(X, d)$, 
and 
the Assouad dimension $\adim(X, d)$ of 
a metric space $(X, d)$.
The definitions of these dimensions will be  presented in 
Section \ref{sec:pre}.
The topological dimension takes values 
in $\zz_{\ge -1}\cup\{\infty\}$, 
and the other 
four dimensions  take  values in $[0, \infty]$.  
For every  bounded metric space $(X, d)$, 
we have 
 the following basic inequalities (see Theorem \ref{thm:diminequalities}):
\begin{align*}\label{eq:basicinq}
\tdim X\le \hdim(X, d)\le \pdim(X, d)\le \ubdim(X, d)\le \adim(X, d). 
\end{align*}

\subsection{Fractal dimensions}
A topological space is said to be a \emph{Cantor space} if it is homeomorphic to the Cantor set. 
A metric $d: X^{2}\to [0, \infty)$ on a set $X$ 
is said to be a \emph{non-Archimedean metric} 
or \emph{ultrametric} if 
for all $x, y, z\in X$ the  strong triangle inequality:  
$d(x, y)\le d(x, z)\lor d(z, y)$ is satisfied, where $\lor$ is the 
maximal operator on $\rr$. 

In \cite{Ishiki2019}, 
the author proved that 
for all $a, b\in [0, \infty]$ with 
$a\le b$, there exists a Cantor metric space 
$(X, d)$ with $\hdim(X, d)=a$ and $\adim(X, d)=b$. 
As a development of this result, 
we solve  the  problems  of prescribed dimensions for 
 the  five  dimensions explained above.

We denote by $\linset$ the set of all 
$(a_{1}, a_{2}, a_{3}, a_{4})\in [0, \infty]^{4}$ satisfying 
\[
a_{1}\le a_{2}\le a_{3}\le a_{4}. 
\] 
We also denote by $\linsettwo$
the set of all 
$(l, a_{1}, a_{2}, a_{3}, a_{4})\in (\zz_{\ge 0}\cup\{\infty\})\times [0, \infty]^{4}$ satisfying 
$l\le a_{1}\le a_{2}\le a_{3}\le a_{4}$.

\begin{thm}\label{thm:prescribed}
For every $(a_{1}, a_{2}, a_{3}, a_{4})\in \linset$, 
there exists a Cantor ultrametric space $(X, d)$ such that 
\[
\hdim(X, d)=a_{1}, \pdim(X, d)=a_{2}, \ubdim(X, d)=a_{3}, 
\adim(X, d)=a_{4}.
\]
\end{thm}

By Theorem \ref{thm:prescribed}, we obtain the following 
(see also Theorem \ref{thm:diminequalities}):
\begin{thm}\label{thm:prescribed2}
For every $(l, a_{1}, a_{2}, a_{3}, a_{4})\in \linsettwo$, 
there exists a compact metric space $(X, d)$ such that 
\begin{align*}
&\tdim X=l, 
\hdim(X, d)=a_{1}, \pdim(X, d)=a_{2}, \ubdim(X, d)=a_{3}, \\
&\adim(X, d)=a_{4}.
\end{align*}
\end{thm}

\subsection{Topological embeddings of the Hilbert cube}
In this paper, 
we denote by 
 $\grsp$  
 the  set of all isometry classes of
  non-empty compact metric spaces, and 
 denote by 
 $\grdis$
the Gromov--Hausdorff distance (the definition will be 
presented in Section \ref{sec:pre}). 
The space 
$(\grsp, \grdis)$ is called   the 
\emph{Gromov--Hausdorff space}.
By abuse of notation, 
we represent an element of $\grsp$ as 
 a pair $(X, d)$ of a set $X$ and a metric $d$ rather than its isometry class.

For $(l, a_{1}, a_{2}, a_{3}, a_{4})\in \linsettwo$,  
we denote by 
$\dimset(l, a_{1}, a_{2}, a_{3}, a_{4})$ the set of all 
compact metric spaces in $\grsp$ satisfying 
\begin{align*}
&\tdim X=l, 
\hdim(X, d)=a_{1}, \pdim(X, d)=a_{2}, \ubdim(X, d)=a_{3}, \\
&\adim(X, d)=a_{4}.
\end{align*}

Note that Theorem \ref{thm:prescribed} implies that  
$\dimset(l, a_{1}, a_{2}, a_{3}, a_{4})\neq \emptyset$ for 
all $(l, a_{1}, a_{2}, a_{3}, a_{4})\in \linsettwo$. 

We denote by $\ultsp$ the set of all compact ultrametric spaces in $\grsp$. 

We also  define 
$\qcube =\prod_{i\in \zz_{\ge 0}}[0, 1]$. 
The space $\qcube$ ( or a space homeomorphic to $\qcube$) is called the \emph{Hilbert cube}. 

In \cite{Ishiki2021branching}, 
the author  defined 
$\mathscr{X}(u, v, w)$ for $(u, v, w)\in \{0, 1, 2\}^{3}$
as sets  of all compact metric spaces satisfying or not satisfying 
the doubling property, the uniform disconnectedness, and   the uniform perfectness, which are 
properties  
appearing  in the David--Semmes theorem \cite[Proposition 15.11]{DS1997}. 
The author \cite{Ishiki2021branching} 
proved that for all compact metric spaces $(X, d)$ and  $(Y, e)$ in $\mathscr{X}(u, v, w)$ with 
$\grdis((X, d), (Y, e))>0$, there exist continuum many 
geodesics connecting $(X, d)$ and $(Y, e)$ passing through $\mathscr{X}(u, v, w)$. 
The construction of such  geodesics induces a topological embedding of the Hilbert cube into $\mathscr{X}(u, v, w)$, and hence $\mathscr{X}(u, v, w)$ has infinite topological dimension. 
In the present paper, as an analogue of   this result, 
we prove  
that an arbitrary compact metrizable space can be topologically embedded into  $\dimset(l, a_{1}, a_{2}, a_{3}, a_{4})$ and  $\ultsp$.

By constructing 
topological embeddings 
of an arbitrary compact metrizable  space 
$H$ into $\dimset(l, a_{1}, a_{2}, a_{3}, a_{4})$ 
and $\ultsp$ which maps 
given $n+1$ points in $H$ into given $n+1$ compact metric spaces, 
we prove the  path-connectedness and the infinite-dimensionality of $\dimset(l, a_{1}, a_{2}, a_{3}, a_{4})$ and  $\ultsp$. 
The existence  of these embeddings is based on 
a
construction of  a metric on a direct sum space
of metric spaces using  amalgamations of metrics. 
\begin{thm}\label{thm:topemb}
Assume that 
$\mathscr{S}=\dimset(l, a_{1}, a_{2}, a_{3}, a_{4})$ for 
some numbers  $(l, a_{1}, a_{2}, a_{3}, a_{4})\in \linsettwo$ or 
$\mathscr{S}=\ultsp$. 
Let $n\in \zz_{\ge 1}$, 
and $H$ a compact metrizable space. 
Take  mutually distinct 
$n+1$ points $\{v_{i}\}_{i=1}^{n+1}$ in 
$H$, and 
let $\{(X_{i}, d_{i})\}_{i=1}^{n+1}$ be compact metric spaces 
 in $\mathscr{S}$ satisfying  that 
 $\grdis((X_{i}, d_{i}), (X_{j}, d_{j}))>0$ for all 
 distinct $i, j$. 
 Then  
there exists a topological 
embedding 
$\Phi: H\to \mathscr{S}$ such that 
$\Phi(v_{i})=(X_{i}, d_{i})$ 
for all 
$i\in \{1, \dots, n+1\}$. 
\end{thm}

In \cite[Lemma 2.18]{Ishiki2021branching}, it is shown
that if $\mathcal{D}$ is any one of $\tdim$, $\hdim$, $\pdim$, $\lbdim$, $\ubdim$,  or $\adim$, then 
the set of all compact metric space $(X, d)$ 
such that $\mathcal{D}(X, d)=\infty$ is 
dense in $\grsp$. 
Using  the same method, 
we find that the set 
$\dimset(l, a_{1}, a_{2}, a_{3}, a_{4})$ is dense in 
$\grsp$.  

In this paper, a topological space is said to be 
\emph{infinite-dimensional} if its topological dimension is infinite. 

As consequences of Theorem \ref{thm:topemb}, 
we obtain the following corollaries.  
\begin{cor}\label{cor:dimset}
For all   $(l, a_{1}, a_{2}, a_{3}, a_{4})\in \linsettwo$, 
the set $\dimset(l, a_{1}, a_{2}, a_{3}, a_{4})$ is 
path-connected,  infinite-dimensional, and  dense  in 
 $\grsp$. 
\end{cor} 
\begin{cor}\label{cor:pathinfult}
The set $\ultsp$ of all compact ultrametric spaces  is
path-connected and  infinite-dimensional
  in   $\grsp$. 
\end{cor}

\begin{rmk}
For a metric space $(X, d)$, 
the dilation map $\lambda\mapsto (X, \lambda d)$ 
$(\lambda\in [0, 1])$ 
is a geodesic connecting the one-point metric space 
and $(X, d)$ (see \cite[(4) in Proposition 1.4]{IT2019iso}). 
From this observation, 
it follows that $(\ultsp, \grdis)$
is path-connected. 
M{\'e}moli, Smith and Wan \cite{memoli2019gromov}
 proved that  $(\ultsp, \grdis)$ is a geodesic space 
(see \cite[Theorem 7.13]{memoli2019gromov}), 
and proved that 
$(\ultsp, \grdis)$ is closed and nowhere dense in 
$\grsp$ (see \cite[Proposition 4.17]{memoli2019gromov}). 
Thus,  the path-connectedness of 
$(\ultsp, \grdis)$ 
 have been already known. 
Our construction of paths is obtained as  a by-product of our topological embeddings of 
an arbitrary compact metrizable space, 
and 
this construction  is  enough  to 
obtain another proof of Qiu's theorem. 
\end{rmk}

\begin{ques}
For every $(l, a_{1}, a_{2}, a_{3}, a_{4})\in \linsettwo$, 
is the  metric space $(\dimset(l, a_{1}, a_{2}, a_{3}, a_{4}), \grdis)$ a geodesic space?
\end{ques}

\subsection{Space of compact  ultrametric spaces}
We provide  applications of  
Corollary \ref{cor:pathinfult} to the non-Archimedean
 Gromov--Hausdorff space. 
 
As a non-Archimedean analogue of the Gromov--Hausdorff distance $\grdis$ on $\grsp$, 
the  \emph{non-Archimedean Gromov--Hausdorff distance} $\ugrdis$ on $\ultsp$ was defined in \cite{Z2005} (the definition will be presented in Subsection \ref{subsec:ultra}). 
 The space $(\ultsp, \ugrdis)$ is called the 
 \emph{non-Archimedean Gromov--Hausdorff space}. 
In this paper, 
for a metric space 
$(X, d)$, 
and for a subset 
$A$ 
of 
$X$, 
we represent 
the restricted metric 
$d|_{A^2}$ 
as 
  the same symbol 
$d$ 
as  the ambient metric 
$d$ until otherwise stated.  
From 
Corollary \ref{cor:pathinfult}, 
 we deduce the relationship between the spaces 
 $(\ultsp, \grdis)$ and $(\ultsp, \ugrdis)$.

\begin{cor}\label{cor:discon}
Let $I_{\ultsp}\colon (\ultsp, \grdis)\to (\ultsp, \ugrdis)$ be the identity map of $\ultsp$. 
Then, $I_{\ultsp}$ is not continuous with respect to 
the 
topologies induced from $\grdis$ and $\ugrdis$. 
\end{cor}
\begin{rmk}
It is known that the inverse $I_{\ultsp}^{-1}\colon(\ultsp, \ugrdis)\to
 (\ultsp, \grdis)$ is $(2^{-1})$-Lipschitz (see 
\cite[Corollary 2.9]{qiu2009geometry}). 
\end{rmk}

The following corollary was first proven  by Qiu \cite[Theorem 4.8]{qiu2009geometry} by constructing concrete examples. 
Our proof is different from Qiu's one. 
\begin{cor}\label{cor:ratio}
For every $c\in [2, \infty)$, 
there exists compact ultrametric spaces 
$(X, d)$ and $(Y, e)$ such that 
\[
c\cdot \grdis((X, d), (Y, e))<\ugrdis((X, d), (Y, e)). 
\]
\end{cor}

The organization of this paper is as follows: 
In Section \ref{sec:pre}, 
we prepare and explain the basic 
definitions and statements on metric spaces. 
The definitions of the five dimensions  and 
$\grdis$  appearing 
in Section \ref{sec:intro} are given. 
In Section \ref{sec:Cantor}, 
we define 
the $(\mathbf{m}, \alpha)$-Cantor ultrametric space, and we give formulas to calculate  dimensions of such an ultrametric space. 
Using that formulas, 
first we  construct spaces stated in  Theorem \ref{thm:prescribed} in cases of $a_{i}\in \{0, 1, \infty\}$ for all $i\in \{1, 2, 3, 4\}$. 
By taking a direct sum of those spaces  in specific cases, 
we prove Theorems \ref{thm:prescribed} and 
\ref{thm:prescribed2}. 
In Section \ref{sec:embed}, 
for each $n\in \zz_{\ge 1}$, 
we construct   topological embeddings of 
an arbitrary compact metrizable space into $\dimset(l, a_{1}, a_{2}, a_{3}, a_{4})$ and $\ultsp$. 
We also discuss 
its applications to the non-Archimedean Gromov--Hausdorff space.

\begin{ac}
The author would like to thank Takumi Yokota  
for 
raising  questions,  
for the many stimulating conversations, and 
for the many  helpful comments. 
The author would also like thank to the 
referee for helpful suggestions and 
comments. 
\end{ac}

\section{Preliminaries}\label{sec:pre}

In this section, 
we prepare 
and explain 
the
basic concepts 
and  
statements 
on metric spaces. 
\subsection{Generalities}
For $k\in \zz$, 
we denote by $\zz_{\ge k}$ the set of all 
integers greater than or equal to $k$. 
For a set $S$, we denote by $\card(S)$ the 
cardinality of $S$.
Let $(X,d)$ be a metric space.
For  $x\in X$ and for  $r\in (0,\infty)$, 
we denote by $B(x, r)$ the closed  ball centered at $x$ with 
 radius $r$. 
For a subset $A$, 
we denote by 
$\delta_d(A)$ 
the diameter of 
$A$.

For two metric spaces $
(X, d)$ 
and 
$(Y, e)$, 
we denote by 
$d\times_{\infty} e$  
the 
$\ell^{\infty}$-product metric  
defined by 
$(d\times_{\infty} e)((x, y), (u, v))=
d(x, u) \lor e(y, v)$. 
Note that 
$d\times_{\infty} e$ 
generates the product topology of 
$X\times Y$. 

In this paper, 
we sometimes use the disjoint union 
$\coprod_{i\in I} X_{i}$ 
of a non-disjoint family 
$\{X_{i}\}_{i\in I}$. 
 Whenever we consider the disjoint  union  
 $\coprod_{i\in I}X_{i}$ of 
 a family 
 $\{X_{i}\}_{i\in I}$ 
 of sets 
 (this family is not necessarily disjoint), 
we  identify the family 
$\{X_{i}\}_{i\in I}$ 
with  its disjoint copy unless otherwise stated. 
If each 
$X_{i}$ 
is a topological space, 
we consider that 
$\coprod_{i\in I}X_{i}$ 
is equipped with 
the direct sum topology.

\subsection{Dimensions}
We explain dimensions appearing in 
Section \ref{sec:intro}. 

\subsubsection{The topological dimension}
In this paper, the \emph{topological dimension} means 
the covering dimension. For a separable metrizable space, the topological dimension is equal to 
the large and small inductive dimensions. 
We refer the readers to 
\cite{HW1948}, \cite{P1975}, 
\cite{Nagata1983}, and  \cite{Cdimension} for the details. 

\subsubsection{The Hausdorff dimension}
Let $(X, d)$ be a metric space. 
For $\delta \in (0,\infty)$, 
we denote by $\mathbf{F}_{\delta}(X, d)$ the set of all subsets of $X$ with diameter smaller than $\delta$. 
For  $s\in [0,\infty)$, and $\delta\in (0, \infty)$, 
we define the measure $\mathcal{H}_{\delta}^{s}$ 
on $X$ as
\[
\mathcal{H}_{\delta}^s(A)=\inf\left\{\, \sum_{i=1}^{\infty}\delta_{d}(A_i)^s\ \middle| \ A\subset\bigcup_{i=1}^{\infty}A_i,\ A_i\in \mathbf{F}_{\delta}(X, d)\,\right\}.
\]
For $s\in (0, \infty)$
we define 
the \emph{$s$-dimensional Hausdorff measure 
 $\mathcal{H}^s$}
on $(X, d)$  as 
$\mathcal{H}^s(A)=
\sup_{\delta\in (0,\infty)}\mathcal{H}_{\delta}^s(A)$. 
We denote by $\hdim(X, d)$ 
the \emph{Hausdorff dimension of $(X, d)$} defined as
\begin{align*}
\hdim(X, d)&=\sup \{\,s\in[0,\infty)\mid 
\mathcal{H}^s(X)=\infty\,\}\\
             &=\inf\{\,s\in [0,\infty)\mid 
             \mathcal{H}^s(X)=0\,\}.
\end{align*}

\subsubsection{The packing dimension}
Let $(X, d)$ be a metric space. 
For a subset $A$ of $X$, and for $\delta\in (0, \infty)$, 
we denote by $\mathbf{Pa}_{\delta}(A)$ the set of all 
finite or countable sequence $\{r_{i}\}_{i=1}^{N} (N\in \zz_{\ge 1}\cup\{\infty\})$ in $(0, \delta)$
 for which 
there exists a sequence $\{x_{i}\}_{i=1}^{N}$ in $A$ such 
that if $i\neq j$, then we have 
$B(x_{i}, r_{i})\cap B(x_{j}, r_{j})=\emptyset$. 
For $s\in [0, \infty)$,  $\delta\in (0, \infty)$,
and  a subset $A$ of $X$,  
we define the quantity  $\widetilde{\mathcal{P}}_{\delta}^{s}(A)$ by 
\[
\widetilde{\mathcal{P}}_{\delta}^{s}(A)=
\sup\left\{\, \sum_{i=1}^{N}r_{i}^{s}\ 
\middle| \ \{r_{i}\}_{i=1}^{N} \in \mathbf{Pa}_{\delta}(A)
\, \right\}. 
\]
We then define the \emph{$s$-dimensional pre-packing measure 
$\widetilde{\mathcal{P}}_{0}^{s}$} on $(X, d)$ as 
$\widetilde{\mathcal{P}}_{0}^{s}(A)
=\inf_{\delta\in (0, \infty)}
\widetilde{\mathcal{P}}_{\delta}^{s}$, 
and 
we define the 
\emph{$s$-dimensional packing measure } 
$\mathcal{P}^{s}$ on 
$(X, d)$
 as
\[
\mathcal{P}^{s}(A)=
\inf\left\{\, \sum_{i=1}^{\infty}\widetilde{\mathcal{P}}_{0}^{s}(S_{i})\ 
\middle|\  A\subset \bigcup_{i=1}^{\infty}S_{i}
\, \right\}. 
\]
We denote by $\pdim(X, d)$ 
the \emph{packing  dimension of $(X, d)$} defined as
\begin{align*}
\pdim(X, d)&=\sup \{\,s\in[0,\infty)\mid 
\mathcal{P}^s(X)=\infty\,\}\\
             &=\inf\{\,s\in [0,\infty)\mid 
             \mathcal{P}^s(X)=0\,\}.
\end{align*}

\begin{rmk}
Our definition of the packing measure is of 
\emph{radius-type}. 
There is another definition of \emph{diameter-type}. 
If we adopt the diameter-type definition, 
then the statement (3) in Theorem \ref{thm:locdim} does not hold true in general. 
For the difference between two definitions, we refer the readers to \cite[Lemma 1.5.7]{joyce1995packing} and \cite{cutler1995density}. 
These two types of the packing measure are 
often treated indistinguishably, usually by imposing a regularity condition such as 
the assumption that  the whole metric space is
a Euclidean space, or 
 geodesic. 
\end{rmk}

\subsubsection{Box dimensions}
For a metric space $(X, d)$, and for $r\in (0, \infty)$, 
 a subset $A$ of $X$ is said to be an 
 \emph{$r$-net} if 
 $X=\bigcup_{a\in A}B(a, r)$. 
 We denote by $\mnum_{d}(X, r)$ the least  cardinality  of  $r$-nets of $X$.  
We define 
the \emph{upper box dimension $\ubdim(X, d)$} and 
the \emph{lower box dimension $\lbdim(X, d)$}  by 
\begin{align*}
\ubdim(X, d)&=\limsup_{r\to 0}\frac{\log \mnum_{d}(X, r)}{-\log r}, \\
\lbdim(X, d)&=\liminf_{r\to 0}\frac{\log \mnum_{d}(X, r)}{-\log r}. 
\end{align*}

\begin{rmk}
In this paper, 
we do not mainly treat the lower box dimension. 
\end{rmk}

\subsubsection{The Assouad dimension}

For a metric space $(X, d)$, 
we denote by  $\adim(X, d)$ 
the \emph{Assouad dimension 
of $(X, d)$}
defined  by the infimum of all 
$\lambda\in (0, \infty)$ for which 
there exists $C\in (0, \infty)$ such that 
for all $R, r\in (0, \infty)$ with $R>r$ and for all $x\in X$
we have 
\[
\mnum_{d}(B(x, R), r)\le 
C\left(\frac{R}{r}\right)^{\lambda}. 
\]
If such $\lambda$ does not exist, we define 
$\adim(X, d)=\infty$. 
We say that a metric space is \emph{doubling} if 
its Assouad dimension is finite. 

We show some properties of the Assouad dimension. 
By the definition of the Assouad dimension, 
we first  obtain:
\begin{lem}\label{lem:upadim}
Let $\lambda\in (0, 1)$, and $N\in \zz_{\ge 2}$. 
Let $(X, d)$ be a metric space. 
If every closed ball with radius $r$ 
can be covered by at most $N$ many closed balls with radius $\lambda r$, then we have 
\[
\adim(X, d)\le \log N/\log \lambda^{-1}.
\] 
\end{lem}

\begin{df}
Let $(X, d)$ be a metric space. 
We define a function $\Theta_{X, d}\colon (0, 1)\to [1, \infty]$
by 
$\Theta_{X, d}(\epsilon)=
\sup_{x\in X}\sup_{R\in (0, \infty)} 
\mnum_{d}(B(x, R), \epsilon R)$. 
\end{df}
Combining the definitions of 
$\Theta_{X, d}(\epsilon)$ and the Assouad dimension, 
we can verify: 

\begin{lem}\label{lem:adimtheta}
Let $(X, d)$ be a metric space. Then, 
the Assouad dimension
$\adim(X, d)$ is equal to
the  infimum of all 
$\lambda\in (0, \infty)$ for which 
there exists $C\in (0, \infty)$ such that 
for all $\epsilon \in (0, 1)$
we have 
$\Theta_{X, d}(\epsilon)\le C\epsilon^{-\lambda}$. 
\end{lem}

Lemma \ref{lem:upadim} implies: 
\begin{lem}\label{lem:doublingiff}
For every  metric space $(X, d)$,  
 the following statements are equivalent to each other. 
\begin{enumerate}
\item The space $(X, d)$ is doubling. \label{item:doubling}
\label{item:adim1} 
\item 
There exists $\epsilon\in (0, 1)$ such that $\Theta_{X, d}(\epsilon)<\infty$. \label{item:doubling2}
\item 
For all $\epsilon\in (0, 1)$, we have 
$\Theta_{X, d}(\epsilon)<\infty$. \label{item:doubling3}

\end{enumerate}
\end{lem}

The next lemma stats that,  to calculate 
the Assouad dimension, 
we only need the information of the 
behavior of  
$\Theta_{X, d}(\epsilon)$ on a restricted  domain. 
\begin{lem}\label{lem:adimcut}
Let $(X, d)$ be a metric space. 
Then $\adim(X, d)$ is equal to  the infimum of all 
 $\lambda\in (0, \infty)$ satisfying that 
there exist $C\in (0, \infty)$ and 
$\epsilon_{0}\in (0, 1)$ such that for all 
$\epsilon \in (0, \epsilon_{0}]$ we have 
$\Theta_{X, d}(\epsilon)\le C\epsilon^{-\lambda}$. 
\end{lem}
\begin{proof}
Let $\beta$ be the infimum stated in the lemma. 
By the definitions of $\beta$ and the Assouad dimension, 
we have $\beta\le \adim(X, d)$. 
To prove the opposite inequality, 
we take 
$\lambda\in (\beta, \infty)$. 
Then, 
there exist $C_{0}\in (0, \infty)$ and 
$\epsilon_{0}\in (0, 1)$ such that every 
$\epsilon \in (0, \epsilon_{0}]$ 
satisfies 
$\Theta_{X, d}(\epsilon)\le C_{0}\epsilon^{-\lambda}$. 
Since 
$\Theta_{X, d}(\epsilon_{0})\le C_{0}\epsilon_{0}^{-\lambda}<\infty$, 
 Lemma 
\ref{lem:doublingiff} shows that
the space $(X, d)$ is doubling. 
Thus there exists $\sigma\in (0, \infty)$ and 
$C_{1}\in (0, \infty)$ such that for all 
$\epsilon \in (0, 1)$ we have $\Theta_{X, d}(\epsilon)\le C_{1}\epsilon^{-\sigma}$. 
Put 
$C=\max\{1, C_{0}, C_{1}\epsilon_{0}^{-\sigma}\}$. 
We obtain 
$\Theta_{X, d}(\epsilon)\le C\epsilon^{-\lambda}$ 
for all $\epsilon \in (0, 1)$. 
Therefore, from Lemma \ref{lem:adimtheta}, 
it follows that 
$\adim(X, d)\le \beta$. 
\end{proof}

Lemmas \ref{lem:doublingiff} and  \ref{lem:adimcut} imply the characterization of 
non-doubling metric spaces:
\begin{cor}\label{cor:infiniteadimcut}
For every metric space $(X, d)$,
 the following statements are equivalent to each other. 
\begin{enumerate}
\item 
The space 
$(X, d)$ is not doubling.  \label{item:adiminf}
\item 
There exists $\epsilon\in (0, 1)$ such that 
$\Theta_{X, d}(\epsilon)=\infty$. 
\item 
For all $\epsilon \in (0, 1)$, we have 
$\Theta_{X, d}(\epsilon)=\infty$. 
\item
There exists a sequence $\{\epsilon_{n}\}_{n\in \zz_{\ge 0}}$ in $(0, 1)$ such that $\epsilon_{n}\to 0$ as $n\to \infty$ and 
$\epsilon_{n}^{-n}< \Theta_{X, d}(\epsilon_{n})$ for all $n\in \zz_{\ge 0}$. \label{item:adiminf3}
\end{enumerate}
\end{cor}

\begin{df}
For a metric space $(X, d)$, 
we define 
$\eta_{X, d}\colon
(0, 1)\to [0, \infty]$ 
by 
\[
\eta_{X, d}(\epsilon)=
\frac{\log \Theta_{X, d}(\epsilon)}{-\log \epsilon}. 
\]
\end{df}

Using the function $\eta_{X, d}$, 
we find a simple formula of 
the Assouad dimension. 
\begin{prop}\label{prop:adimform}
For every metric space $(X, d)$, 
 we have 
\begin{align*}
\adim(X, d)
=
\limsup_{\epsilon \to 0}
\eta_{X, d}(\epsilon). 
\end{align*}
\end{prop}
\begin{proof}
Due to  Corollary \ref{cor:infiniteadimcut}, 
if $\adim(X, d)=\infty$ or $\limsup_{\epsilon \to 0}
\eta_{X, d}(\epsilon)=\infty$, 
then $\adim(X, d)
=
\limsup_{\epsilon \to 0}
\eta_{X, d}(\epsilon)$. 
We next consider  the case of 
$\adim(X, d)<\infty$ and  $\limsup_{\epsilon \to 0}
\eta_{X, d}(\epsilon)<\infty$. 
By the definition of $\adim(X, d)$, 
for all $\lambda\in (\adim(X, d), \infty)$, 
we can find $C\in (0, \infty)$ such that  for all 
$\epsilon \in (0, 1)$ we have 
$\Theta_{X, d}(\epsilon)\le C\epsilon^{-\lambda}$. 
Then we obtain  $\limsup_{\epsilon\to 0}\eta_{X, d}(\epsilon)\le \lambda$, and hence 
$\limsup_{\epsilon\to 0}\eta_{X, d}(\epsilon)
\le \adim(X, d)$. 
To prove the opposite  inequality, 
take $\lambda\in [0, \infty)$ with 
$\limsup_{\epsilon\to 0}\eta_{X, d}(\epsilon)<\lambda$. 
In this case, 
there exists $\epsilon_{0}\in (0, 1)$ such that 
for all $\epsilon\in (0, \epsilon_{0}]$ we have 
$\eta_{X, d}(\epsilon)\le \lambda$. 
Thus
every  $\epsilon\in (0, \epsilon_{0}]$
satisfies 
$\Theta_{X, d}(\epsilon)\le \epsilon^{-\lambda}$.
Lemma \ref{lem:adimcut} implies that 
$\adim(X, d)\le \lambda$, and hence 
$\adim(X, d)\le \limsup_{\epsilon\to 0}\eta_{X, d}(\epsilon)$. 
\end{proof}

\subsection{Properties of dimensions}

Let $f\colon(X, d)\to (Y, e)$ be a map between metric spaces. 
Let  $L\in [1, \infty)$ and $\gamma\in (0,\infty)$. 
We say that $f$ is \emph{$(L, \gamma)$-homogeneously bi-H\"older} if  
for all $x, y\in X$ we have 
\[
L^{-1}d(x, y)^{\gamma}\le e(f(x), f(y))\le Ld(x, y)^{\gamma}.
\]
If $\gamma=1$, the map $f$
 is said to be  \emph{$L$-bi-Lipschitz}. 
\begin{rmk}
For $\gamma>1$, there exists a metric space $(X, d)$ such that the function $d^{\gamma}$ defined by 
$d^{\gamma}(x, y)=(d(x, y))^{\gamma}$ is also a metric 
(for example, an ultrametric $d$). 
Then,  (bi-)H\"{o}lder maps with exponent $\gamma>1$ between general metric spaces  are  
 meaningful. 
\end{rmk}

We refer the readers to 
\cite{falconer1997techniques}, 
\cite{falconer2004fractal}, and \cite{fraser2020assouad} for the details of the following. 
\begin{prop}\label{prop:propdim}
Let $(X, d)$ be a metric space. 
Let $\mathcal{D}$ stand for any one of 
$\hdim$, $\ubdim$,  $\pdim$, or $\adim$. 
Then the following  hold true: 
\begin{enumerate}
\item 
For every subset $A$ of $X$, 
we have $\mathcal{D}(A, d)\le \mathcal{D}(X, d)$. \label{item:monotone}
\item Let $A, B$ be subsets of $X$ with $A\cup B=X$. 
Then we have $\mathcal{D}(X, d)=\max\{\mathcal{D}(A, d), \mathcal{D}(B, d)\}$.\label{item:finitestab} 
\item Let $(Y, e)$ be a metric space, 
and $f\colon (X, d)\to (Y, e)$ be an $(L, \gamma)$-homogeneously bi-H\"{o}lder map. 
Then,  we obtain the equality  $\mathcal{D}(f(X), e)=
(1/\gamma)\mathcal{D}(X, d)$. \label{item:biholder}
\item If $\mathcal{D}$ is either  $\hdim$ or $\pdim$, 
then every countable covering $\{S_{i}\}_{\in \zz_{\ge 1}}$ of $X$ satisfies 
$\mathcal{D}(X, d)=\sup_{i\in \zz_{\ge 1}}\mathcal{D}(S_{i}, d)$. \label{item:countablestab}
\end{enumerate}
\end{prop}

\begin{thm}\label{thm:diminequalities}
Let $(X, d)$ be a  bounded metric space. 
Then, 
\[
\tdim X \le \hdim(X, d)\le \pdim(X, d)\le \ubdim(X, d)\le \adim(X, d). 
\]

\end{thm}
\begin{proof}
The inequality $\tdim X \le \hdim(X, d)$ is due to 
Szpilrajn \cite{Szpilrajn1937} (see also \cite{H2001}). 
The inequality $\ubdim(X, d)\le \adim(X, d)$ is proven in 
\cite[Lemma 2.4.3]{fraser2020assouad}, or we can prove it using Proposition \ref{prop:adimform}. 
The proofs of 
the remaining inequalities are  presented in 
\cite[Chapters 2 and  3]{falconer2004fractal}, 
which proofs are  valid for not only  subsets of the Euclidean spaces but also general metric spaces. 
\end{proof}

For the details of the next proposition, we refer the readers to 
 \cite{HW1948}, \cite{P1975}, 
\cite{Nagata1983}, and  \cite{Cdimension}. 
\begin{prop}\label{prop:tdimprod}
Let $X$ and $Y$ be separable metrizable spaces. 
Then we have 
$\tdim(X\sqcup Y)=\max\{\tdim(X), \tdim(Y)\}$ and 
$\tdim(X\times Y)\le \tdim X+\tdim Y$. 
\end{prop}

The proof of the following is presented  in \cite{falconer2004fractal} and 
\cite{fraser2020assouad}. 
\begin{prop}\label{prop:productdim}
Let $(X, d)$ and $(Y, e)$ be metric spaces. 
Then the following statements hold true. 
\begin{enumerate}
\item 
Let $\mathcal{D}$ be any one of $\pdim$, 
$\ubdim$, or $\adim$, then we have 
\[
\mathcal{D}(X\times Y, d\times_{\infty}e)\le 
\mathcal{D}(X, d)+\mathcal{D}(Y, e).
\] 
\item 
$\hdim(X\times Y, d\times_{\infty}e)\le \hdim(X, d)+\pdim(Y, e)$. \label{item:haus2}
\end{enumerate}
\end{prop}
\begin{rmk}
In general, we can not replace $\pdim$ with $\hdim$
in the right-hand side in   (\ref{item:haus2}) of Proposition \ref{prop:productdim}. 
Indeed, 
for all metric spaces $(X, d)$ and $(Y, e)$, 
we have 
$\hdim(X, d)+\hdim(Y, e)\le \hdim(X\times Y, d\times_{\infty}e)$.
For the details, we refer the readers to  \cite{falconer2004fractal}
and \cite{fraser2020assouad}. 
\end{rmk}

\subsection{The Gromov--Hausdorff distance}
For a metric space 
$(Z, h)$,  
and  for subsets 
$A$,  $B$ 
of 
$Z$, 
we denote by 
$\hdis(A, B; Z, h)$ 
the \emph{Hausdorff distance} of 
$A$ 
and 
$B$ 
in 
$Z$. 
For  metric spaces 
$(X, d)$ 
and 
$(Y, e)$, 
the 
\emph{Gromov--Hausdorff distance} 
$\grdis((X, d),(Y, e))$ between 
$(X, d)$ 
and 
$(Y, e)$ 
is defined as 
the infimum of  all values  
$\hdis(i(X), j(Y); Z, h)$, 
where 
$(Z, h)$ 
is a metric space, 
 and 
$i\colon X\to Z$ 
and 
$j\colon Y\to Z$ 
are isometric embeddings. 
Let $f\colon X\to Y$ be a map between metric spaces
 $(X, d)$ and $(Y, e)$. 
We define the \emph{distortion $\dis(f)$  of $f$} by 
\[
\dis(f)=\sup_{x, y\in X}|d(x, y)-e(f(x), f(y))|. 
\]
The following is deduced from  
\cite[Corollary 7.3.28]{BBI}. 
\begin{lem}\label{lem:disfunc}
Let $(X, d)$ and $(Y, e)$ be  compact metric spaces, 
and  $f\colon X\to Y$ be  a surjective  map. 
Then
$\grdis((X, d), (Y, e))\le 2\dis(f)$. 
\end{lem}

For a set $X$, 
a map $d\colon  X\times X\to [0, \infty)$ is said to be 
a \emph{pseudo-metric} if 
$d$ satisfies the triangle inequality and satisfies 
that $d(x, x)=0$  and $d(x, y)=d(y, x)$ for all $x, y\in X$.
If 
a pseudo-metric $d$ satisfies that  $d(x, y)=0$ implies $x=y$, then $d$ is a metric. 
A pseudo-metric is said to be a 
\emph{pseudo-ultrametric} if it satisfies the strong triangle inequality. 
We denote by $\pmet(X)$ (resp.~$\pumet(X)$) 
the set of 
all pseudo-metrics (resp.~pseudo-ultrametrics) 
on $X$. 
We define a metric  $\mathcal{D}_{X}$ on 
 $\pmet(X)$ by 
 $\mathcal{D}_{X}(d, e)=
 \sup_{x, y\in X}|d(x, y)-e(x, y)|$. 
 Note that $\mathcal{D}_{X}$ 
can  take the value $\infty$. 

Let $d\in \pmet(X)$. 
We denote by $X_{/d}$ the quotient set 
by the relation $\sim_{d}$ defined by 
$x \sim_{d} y\iff d(x, y)=0$.  
We also denote by $[x]_{d}$ the equivalence class of $x$ by 
$\sim_{d}$. We define a metric $[d]$ on $X_{/d}$ by 
$[d]([x]_{d}, [y]_{d})=d(x, y)$. 
The metric $[d]$ is well-defined. 
Remark that if $d$ is a metric, then 
$(X_{/d}, [d])$ is isometric to $(X, d)$. 
\begin{prop}\label{prop:ghquotient}
Let  $d, e\in \pmet(X)$, and 
assume that $e$ is a metric on $X$. 
Then 
we have 
$\grdis\left(\left(X_{/d}, [d]\right), \left(X, e\right)\right)\le 2\cdot  \mathcal{D}_{X}(d, e)$. 
\end{prop}
\begin{proof}
Let $p\colon X\to X_{/d}$ be the canonical projection. 
Since  $\dis(p)= \mathcal{D}_{X}(d, e)$, 
the proposition follows from Lemma \ref{lem:disfunc}. 
\end{proof}
Proposition \ref{prop:ghquotient} implies:
\begin{cor}\label{cor:metconti}
Let $T$ be a topological space all of whose
finite subsets are closed. 
Let $X$ be a set. 
If a map $h\colon T\to \pmet(X)$  is continuous and 
there exists a finite subset $L$ of $T$ such that 
$h(t)$ is a metric for all $t\in T\setminus L$, 
then the map $F\colon T\to \grsp$ defined by 
$F(t)=(X_{/h(t)}, [h(t)])$ is continuous. 
\end{cor}

\subsection{Amalgamations of pseudo-metrics}
For a topological  space $X$, 
we denote by 
$\met(X)$  (resp.~$\ult(X)$)
 the set of all metrics (resp.~all ultrametrics) 
 on $X$  generating the same topology of $X$. 
 For every $n\in \zz_{\ge 1}$, 
 we denote by $\widehat{n}$
the set $\{1, \dots, n\}$. 
In what follows, we consider that 
the set $\widehat{n}$ 
is always equipped with the discrete topology. 

Since 
Proposition 3.1 in \cite{ishiki2021dense}
 treats 
a similar construction, 
we omit the proofs of the following lemmas. 
\begin{lem}\label{lem:amal2}
Let $n\in \zz_{\ge 2}$. 
Let $\{X_{i}\}_{i=1}^{n}$ be metrizable spaces and 
$\{d_{i}\}_{i=1}^{n}$ be pseudo-metrics with 
$d_{i}\in \pmet(X_{i})$. 
Let $r\in \pmet(\widehat{n})$
and 
$p_{i}\in X_{i}$. 
We define a symmetric function $D$ on 
$(\coprod_{i=1}^{n}X_{i})^{2}$ by 
\[
D(x, y)=
\begin{cases}
d_{i}(x, y) & \text{if $x, y\in X_{i}$;}\\
d_{i}(x, p_{i})+r(i, j)+d_{j}(p_{j}, y) & \text{if $x\in X_{i}$ and $y\in X_{j}$. }
\end{cases}
\]
Then
$D\in \pmet(\coprod_{i=1}^{n}X_{i})$. 
Moreover, 
if each $d_{i}$ is in $\met(X_{i})$, and if 
$r\in \met(\widehat{n})$, then 
$D\in \met(\coprod_{i=1}^{n}X_{i})$. 
\end{lem}

\begin{lem}\label{lem:ultraamal}
Let $n\in \zz_{\ge 2}$. 
Let $\{X_{i}\}_{i=1}^{n}$ be ultrametrizable spaces and 
$\{d_{i}\}_{i=1}^{n}$ be pseudo-ultrametrics with 
$d_{i}\in \pumet(X_{i})$. 
Let $r\in \pumet(\widehat{n})$ and 
$p_{i}\in X_{i}$. 
We define a symmetric function $D$ on 
the set 
$(\coprod_{i=1}^{n}X_{i})^{2}$ by 
\[
D(x, y)=
\begin{cases}
d_{i}(x, y) & \text{if $x, y\in X_{i}$;}\\
d_{i}(x, p_{i})\lor r(i, j)\lor d_{j}(p_{j}, y) & \text{if $x\in X_{i}$ and $y\in X_{j}$. }
\end{cases}
\]
Then
$D\in \pumet(\coprod_{i=1}^{n}X_{i})$. 
Moreover, if each $d_{i}$ is in 
$\ult(X_{i})$, and if 
$r\in \ult(\widehat{n})$, then 
$D\in \ult(\coprod_{i=1}^{n}X_{i})$. 
\end{lem}

\section{Prescribed dimensions}\label{sec:Cantor}

\subsection{Cantor ultrametric spaces}
\begin{df}
Let $\numset$ be the set of all sequences 
 $\mathbf{m}=\{m_{i}\}_{i\in \zz_{\ge 0}}$ of integers  with $m_{i}\ge 2$ for all $i\in \zz_{\ge 0}$. 
 In other words, 
 $\numset=(\zz_{\ge 2})^{\zz_{\ge 0}}$. 
A map  $\alpha\colon \zz_{\ge 0}\to (0, \infty)$ is 
said to be  a \emph{shrinking sequence} if $\alpha$ is 
strictly decreasing 
 and converges to $0$.
 We denote by $\shrinkset$ the set of all shrinking sequences. 
Let $S(\mathbf{m})$ be the set of all maps $x$ from 
$\zz_{\ge 0}$ into $\zz_{\ge 0}$ such that 
$x(i)\in \{0, 1, \dots, m_{i}-1\}$ for all $i\in \zz_{\ge 0}$. 
We  define a valuation $v\colon S(\mathbf{m})\times S(\mathbf{m})\to \zz_{\ge 0}\sqcup\{\infty\}$ by 
$v(x, y)=\min\{\, n\in \zz_{\ge 0}\mid x(n)\neq y(n)\, \}$ if $x\neq y$; otherwise $v(x, y)=\infty$. 
We also define 
$\alpha_{\sharp}(x, y)=\alpha(v(x, y))$, where 
we put $\alpha(\infty)=0$. 
Then $\alpha_{\sharp}$ is an ultrametric on 
$S(\mathbf{m})$. 
Notice that 
$(S(\mathbf{m}), \alpha_{\sharp})$ is a 
Cantor space. 
In this paper, 
the space  $(S(\mathbf{m}), \alpha_{\sharp})$ is called 
the \emph{$(\mathbf{m}, \alpha)$-Cantor ultrametric space}. 
This space is a generalization of sequentially metrized Cantor spaces defined in the author's paper 
 \cite{Ishiki2019}. 
This  construction of Cantor spaces  has been  utilized in fractal geometry (for example, \cite{joyce1995packing}). 
\end{df}

For $\alpha\in \shrinkset$, and for 
$\mathbf{m}=\{m_{i}\}_{i\in \zz_{\ge 0}}\in \numset$,  and for  $k\in \zz_{\ge 1}$,  
we define the \emph{$k$-shifted shrinking sequence $\alpha^{\{k\}}$ of $\alpha$} by $\alpha^{\{k\}}(n)=\alpha(n+k)$, and 
define \emph{$k$-shifted  sequence 
$\mathbf{m}^{\{k\}}=\{m_{i}^{\{k\}}\}_{i\in \zz_{\ge 0}}$ of $\mathbf{m}$} by 
${m}^{\{k\}}_{i}={m}_{i+k}$. 

From the definition of $k$-shifted sequences, 
we deduce  the following lemma.  
\begin{lem}\label{lem:shiftedCantor}
Let $\mathbf{m}\in \numset$, and 
let $\alpha\in \shrinkset$. 
Let  $k\in \zz_{\ge 0}$. 
Then the metric  space 
$\left(S(\mathbf{m}^{\{k\}}),
(\alpha^{\{k\}})_{\sharp}\right)$ is isometric to the closed ball $B(x,\alpha(k))$ in $(S(\mathbf{m}), \alpha_{\sharp})$ for all $x\in S(\mathbf{m})$. 
\end{lem}

\begin{lem}\label{lem:numCantor}
Let  
$\mathbf{m}=\{m_{i}\}_{i\in \zz_{\ge 0}}\in \numset$ 
and 
$\alpha\in \shrinkset$. 
If $r\in (0, \infty)$ and $n\in \zz_{\ge 0}$ satisfy 
that $\alpha(n+1)\le r< \alpha(n)$, then 
we have 
\[
\mnum_{\alpha_{\sharp}}(S(\mathbf{m}), r)=
m_{0}\cdots m_{n}. 
\]
Moreover, 
if $R, r\in (0, \infty)$ and $n, k\in \zz_{\ge 0}$ satisfy $R>r$, $\alpha(n+1)\le R<\alpha(n)$,  and 
$\alpha(n+k+1)\le r<\alpha(n+k)$, then 
\[
\mnum_{\alpha_{\sharp}}(B(x, R), r)=
\begin{cases}
1 & \text{if $k=0$;}\\
m_{n+1}\cdots m_{n+k} & \text{otherwise.}
\end{cases}
\]
\end{lem}
\begin{proof}
Let $A$ be the set of all $x\in S(\mathbf{m})$ with 
$x(k)=0$ for all $k\in\zz_{\ge n+1}$. 
Then $A$ is a minimal  $r$-net, and hence  
$\mnum_{\alpha_{\sharp}}(S(\mathbf{m}), r)=
\card(A)=m_{0}\cdots m_{n}$. 
The latter part follows from Lemma \ref{lem:shiftedCantor}. 
\end{proof}

The
upper and lower box dimensions 
of  the $(\mathbf{m}, \alpha)$-Cantor ultrametric space
can be  computed  using 
$\mathbf{m}$ and $\alpha$.

\begin{prop}\label{prop:boxdimCantor}
For all $\mathbf{m}=\{m_{i}\}_{i\in \zz_{\ge 0}}\in \numset$ and 
 $\alpha \in \shrinkset$, 
we have 
\begin{align}
&\lbdim(S(\mathbf{m}), \alpha_{\sharp})=
\liminf_{n\to\infty}\frac{\log(m_{0}\cdots m_{n})}{-\log \alpha(n+1)}, \\
&\ubdim(S(\mathbf{m}), \alpha_{\sharp})=
\limsup_{n\to\infty}\frac{\log(m_{0}\cdots m_{n})}{-\log \alpha(n)}.
\end{align}
\end{prop}
\begin{proof}
Take $N\in \zz_{\ge 0}$ with $\alpha(N)<1$. 
Take  $r\in (0, \alpha(N))$, and  
 $n\in \zz_{\ge N}$ with
$\alpha(n+1)\le r< \alpha(n)$. 
Then,  by $\alpha(n)<1$, 
\[
\frac{1}{-\log\alpha(n+1)}\le  \frac{1}{-\log r}<
\frac{1}{-\log \alpha(n)}. 
\]
From  these inequalities,  and Lemma \ref{lem:numCantor},  and the definitions of 
the  lower and upper box dimensions, Proposition \ref{prop:boxdimCantor} follows. 
\end{proof}

To calculate the Hausdorff dimension and 
the packing dimension, we use the local dimensions of 
measures on metric spaces.
Let $(X, d)$ be a separable metric space, 
and $\mu$ be a finite
Borel   measure on $X$. 
For every $x\in X$, 
we define the upper (reps.~lower) local dimension 
$\ulocdim\mu(x)$ (resp.~$\llocdim\mu(x)$) by 
\begin{align}
&\ulocdim\mu(x)=\limsup_{r\to 0}
\frac{\log\mu(B(x, r))}{\log r}, \\
&\llocdim\mu(x)=\liminf_{r\to 0}
\frac{\log \mu(B(x, r))}{\log r}. 
\end{align}

The following theorem is well-known. 
The proofs  on the Hausdorff dimension, 
and the packing dimension  are presented in 
\cite[Lemmas 2.1 and 2.2]{cutler1990connecting}, and 
\cite[Corollary 3.20]{cutler1995density}, respectively. 
The paper 
\cite[Corollary 2.9]{olsen1995multifractal} provides its sophisticated version.  
Proposition 2.3 in the book \cite{falconer1997techniques} 
 treats Theorem \ref{thm:locdim} only  in the Euclidean setting; 
 however,  that proof is also valid  for  the general setting
  since the so-called ``$5r$ covering lemma'' holds true in general metric spaces (see \cite[Theorem 1.2]{H2001}).

\begin{thm}\label{thm:locdim}
Let $(X, d)$ be a separable metric space, 
and $\mu$ be a finite
Borel   measure on $X$. 
Let $s\in [0, \infty)$. 
Then, we obtain: 
\begin{enumerate}
\item 
If $s\le \llocdim\mu(x)$ for all $x\in X$
and $\mu(X)>0$, 
then we have 
$s\le \hdim(X, d)$. 
\item 
If $\llocdim\mu(x)\le s$ for all $x\in X$, 
then we have 
$\hdim(X, d)\le s$. 
\item 
If $s\le \ulocdim\mu(x)$ for all $x\in X$
and $\mu(X)>0$, 
then we have 
$s\le \pdim(X, d)$. 
\item 
If $\ulocdim\mu(x)\le s$ for all $x\in X$, 
then  we have
$\pdim(X, d)\le s$. 

\end{enumerate}

\end{thm}

\begin{df}\label{df:distribution}
Let $\mathbf{m}=\{m_{i}\}_{i\in \zz_{\ge 0}}\in \numset$ and 
$\alpha\in \shrinkset$. 
We denote by $\mu_{\mathbf{m}, \alpha}$ 
the
probability  Borel 
measure on  $(S(\mathbf{m}), \alpha_{\sharp})$ satisfying  
\[
\mu_{\mathbf{m}, \alpha}(B(x, \alpha(n+1)))=\frac{1}{m_{0}\cdots m_{n}}
\]
for all $x\in S(\mathbf{m})$ and $n\in \zz_{\ge0}$. 
Note that $\mu_{\mathbf{m}, \alpha}$ always exists 
since 
it is the countable  product of uniform measures on $\{0, 1, \dots, m_i-1\}$
(see also the argument of  ``repeated subdivision'' in 
\cite{falconer1997techniques} and \cite{falconer2004fractal}). 
\end{df}

Due to the measure $\mu_{\mathbf{m}, \alpha}$, 
we obtain the formulas to calculate the 
Hausdorff and packing dimensions of $(S(\mathbf{m}), \alpha_{\sharp})$. 
\begin{prop}\label{prop:hpCantor}
For all  
$\mathbf{m}=\{m_{i}\}_{i\in \zz_{\ge 0}}\in \numset$ 
and  $\alpha \in \shrinkset$, 
we have 
\begin{align}
&\hdim(S(\mathbf{m}), \alpha_{\sharp})=
\liminf_{n\to\infty}\frac{\log(m_{0}\cdots m_{n})}{-\log \alpha(n+1)}, \\
&\pdim(S(\mathbf{m}), \alpha_{\sharp})=
\limsup_{n\to\infty}\frac{\log(m_{0}\cdots m_{n})}{-\log \alpha(n)}. 
\end{align}

\end{prop}

\begin{proof}
Take $N\in \zz_{\ge 0}$ with $\alpha(N)<1$. 
Take $r\in (0, \alpha(N))$, and 
 $n\in \zz_{\ge N}$ with
$\alpha(n+1)\le r< \alpha(n)$. 
Take arbitrary $x\in S(\mathbf{m})$. 
Then  $B(x, r)=B(x, \alpha(n+1))$,  and 
$\mu_{\mathbf{m}, \alpha}(B(x, r))=1/(m_{0}\cdots m_{n})$. 
By $\alpha(n)<1$, we obtain 
\[
\frac{1}{-\log\alpha(n+1)}\le  \frac{1}{-\log r}<
\frac{1}{-\log\alpha(n)}
\]
Thus,  for all $x\in S(\mathbf{m})$ we obtain 
\begin{align}
\llocdim\mu_{\mathbf{m}, \alpha}(x)&=
\liminf_{n\to\infty}\frac{\log(m_{0}\cdots m_{n})}{-\log \alpha(n+1)}, \\
\ulocdim\mu_{\mathbf{m}, \alpha}(x)&=
\limsup_{n\to\infty}\frac{\log(m_{0}\cdots m_{n})}{-\log \alpha(n)}. 
\end{align}
Thus, Proposition \ref{prop:hpCantor} implies the proposition. 
\end{proof}

\begin{rmk}\label{rmk:equal}
From  Propositions \ref{prop:boxdimCantor} and 
\ref{prop:hpCantor}, 
it follows that for all $\mathbf{m}\in \numset$ and $\alpha\in \shrinkset$, 
we have 
$\lbdim(S(\mathbf{m}), \alpha_{\sharp})=
\hdim(S(\mathbf{m}), \alpha_{\sharp})$
and 
$\ubdim(S(\mathbf{m}), \alpha_{\sharp})=
\pdim(S(\mathbf{m}), \alpha_{\sharp})$.
\end{rmk}

\subsection{Proof of Theorems \ref{thm:prescribed} and \ref{thm:prescribed2}}
\begin{df}
For $(a_{1}, a_{2}, a_{3}, a_{4})\in \linset$, 
we say that a metric space
$(X, d)$
 is \emph{of the  dimensional type $(a_{1}, a_{2}, a_{3}, a_{4})$} if  we have 
\[
\hdim(X, d)=a_{1}, \pdim(X, d)=a_{2}, \ubdim(X, d)=a_{3}, 
\adim(X, d)=a_{4}.
\]
\end{df}

To prove  Theorems \ref{thm:prescribed} and \ref{thm:prescribed2}, 
we construct 
Cantor ultrametric spaces of the dimensional type 
$(a_{1}, a_{2}, a_{3}, a_{4})$ with 
$a_{i}\in \{0, 1, \infty\}$ for all $i\in \widehat{4}$
using Propositions 
\ref{prop:boxdimCantor}
and  \ref{prop:hpCantor}. 

\begin{df}
We denote by $\mathbf{2}$ the sequence in $\numset$
all of whose  entries are equal to $2$. 
\end{df}

Using  Lemma \ref{lem:upadim}, 
we find:
\begin{lem}\label{lem:upadim2}
If $\alpha\in \shrinkset$ satisfies 
$\alpha(i+1)/\alpha(i)\le 2^{-1}$ for all $i\in \zz_{\ge 0}$, 
then 
we have $\adim(S(\mathbf{2}), \alpha_{\sharp})\le 1$. 
\end{lem}

Let $f, g\colon [0, \infty)\to [0, \infty)$ be
arbitrary  functions. 
We write 
$f(\epsilon)=O(g(\epsilon))(\epsilon \to 0)$ if 
there exist $M\in (0, \infty)$
and  $\epsilon_{0}\in [0, \infty)$ such that 
every  $\epsilon \in [0, \epsilon_{0})$ satisfies 
$f(\epsilon)\le M\cdot g(\epsilon)$.

\begin{lem}\label{lem:0000}
For all $\mathbf{m}\in \numset$ and $\alpha\in \shrinkset$, 
the space  $(S(\mathbf{m}), \alpha_{\sharp})$ contains 
 a Cantor  subspace  of  the dimensional type
$(0, 0, 0, 0)$. 
\end{lem}
\begin{proof}
First, we can find   a sequence  $\{k(i)\}_{i\in \zz_{\ge 0}}$ in 
$\zz_{\ge 0}$ satisfying  that
\begin{enumerate}[label=\textup{(\arabic*)}]
\item $\alpha(k(0))<1$; 
\item for all $i\in \zz_{\ge 0}$, 
we have
$k(i)<k(i+1)$; 
\item\label{item:37333} 
if $s, t, i\in \zz_{\ge 0}$ satisfy $s+t\le i$, 
then 
$\alpha(k(i+1))<2^{-s^{2}}\alpha(k(t))$. 
\end{enumerate}
We define a subset  $T$ of $S(\mathbf{m})$ by 
the set of all $x\in S(\mathbf{m})$ such that 
$x(n)\in \{0, 1\}$ if $n=k(i)$ for some $i\in \zz_{\ge 0}$; otherwise $x(n)=0$. 
We define $\beta\in \shrinkset$ by 
$\beta(i)=\alpha(k(i))$. 
Then the space $(T, \alpha_{\sharp})$ is isometric to 
$(S(\mathbf{2}), \beta_{\sharp})$. 
Due to  Theorem \ref{thm:diminequalities}, 
it suffices to show that 
$\adim(S(\mathbf{2}), \beta_{\sharp})=0$. 
Take $R, \epsilon\in (0, 1)$.
Let $n, m, N\in \zz_{\ge 0}$  be integers 
satisfying that 
$\beta(n+1)\le R<\beta(n)$, 
$2^{-m^{2}}\le \epsilon <2^{-(m-1)^{2}}$, 
 and 
$\beta(N+1)\le \epsilon R<\beta(N)$. 
Since $2^{-m^{2}}\beta(n+1)\le \beta(N)$, 
using the property  \ref{item:37333}  of $\{k_{i}\}_{i\in \zz_{\ge 0}}$, 
we obtain $N<m+n+2$. 
Lemma \ref{lem:numCantor} implies 
$\mnum_{\beta_{\sharp}}(B(x, R), \epsilon)=
2^{N-n}\le 2^{m+2}$ for all $x\in S(\mathbf{2})$. 
Then
$\log \Theta_{S(\mathbf{2}), \beta_{\sharp}}(\epsilon)
=
O(\sqrt{\log \epsilon^{-1} })\ (\epsilon\to 0)$. 
Hence
$\limsup_{\epsilon\to 0}
\eta_{S(\mathbf{2}), \beta_{\sharp}}(\epsilon)=0$. 
Therefore,  Proposition \ref{prop:adimform}
shows  that 
$\adim(S(\mathbf{2}), \beta_{\sharp})=0$. 
\end{proof}

\begin{lem}\label{lem:1111}
There exist 
$\mathbf{m}\in \numset$ and $\alpha\in \shrinkset$
such that the space  $(S(\mathbf{m}), \alpha_{\sharp})$ is 
of the dimensional type  $(1, 1, 1, 1)$. 
\end{lem}
\begin{proof}
We define $\alpha \in \shrinkset$ by 
$\alpha(n)=2^{-1}$. 
Then, 
combining  
Lemma \ref{lem:upadim2}, 
Proposition \ref{prop:hpCantor}, 
and  Theorem \ref{thm:diminequalities},  
we notice that the space $(S(\mathbf{2}), \alpha_{\sharp})$ is of the dimension type $(1, 1, 1, 1)$ (see also \cite[Lemma 8.8]{Ishiki2019}). 
\end{proof}

\begin{lem}\label{lem:iiii}
There exist 
$\mathbf{m}\in \numset$ and $\alpha\in \shrinkset$ 
such that the space  $(S(\mathbf{m}), \alpha_{\sharp})$ is 
of the dimensional type  $(\infty, \infty, \infty, \infty)$. 
\end{lem}
\begin{proof}
We define $\alpha \in \shrinkset$ by 
$\alpha(n)=(n+1)^{-1}$. 
Then, by Proposition \ref{prop:hpCantor} and 
Theorem \ref{thm:diminequalities}, 
 the space $(S(\mathbf{2}), \alpha_{\sharp})$ is a desired one. 
\end{proof}

Before proving the next lemma, 
we introduce two asymptotic notations of 
functions. 
Let $f, g\colon \zz_{\ge 0}\to (0, \infty)$ be arbitrary maps. We write 
$f(i)\asymp g(i)$ if 
there exist $M_{0}, M_{1}\in (0, \infty)$ and 
$N\in \zz_{\ge 0}$ such that 
for all $i\in \zz_{\ge 0}$ with $N<i$, 
we have $M_{0}\cdot g(i)\le f(i)\le M_{1}\cdot g(i)$. 
We also write $f(i)=o(g(i))(i\to \infty)$ if 
$\lim_{i\to \infty} f(i)/g(i)=0$. 
\begin{lem}\label{lem:0111}
There exist 
$\mathbf{m}\in \numset$ and $\alpha\in \shrinkset$ 
such that the space  $(S(\mathbf{m}), \alpha_{\sharp})$ is 
of the dimensional type $(0, 1, 1, 1)$. 
\end{lem}
\begin{proof}
Take sequences 
 $\{k(i)\}_{i\in \zz_{\ge 0}}$ in $\zz_{\ge 1}$ and 
$\{c_{i}\}_{i\in \zz_{\ge 0}}$ in $(0, 1)$ with 
\begin{enumerate}
\item 
 for every $i\in \zz_{\ge 0}$,  we have 
 $k(i)+1<k(i+1)$;\label{item:kimono}
\item  
 for  every $i\in \zz_{\ge 0}$, we have  
$c_{i}\le 2^{-1}$; \label{item:inqci1/2}
\item
if $n\in \zz_{\ge 0}$ satisfies  $k(i)+2\le n\le k(i+1)$, 
we have $c_{n}=2^{-1}$;\label{item:eqci1/2}
\item 
the following equalities hold true:\label{item:end}
\begin{align*}
\text{(4.A)}\  \ & \lim_{i \to \infty}\frac{k(i)}
{-\log(c_{0}\cdots c_{k(i)+1})}=0;\\
\text{(4.B)}\  \ & \lim_{i\to \infty}\frac{k(i)}{k(i+1)-k(i)-1}=0;\\
\text{(4.C)}\ \ & \lim_{i\to \infty}
\frac{-\log(c_{0}\dots c_{k(i)+1})}{k(i+1)-k(i)-1}=0. 
 \end{align*}
\end{enumerate}

For example,  
we define $\{t(n)\}_{n\in \zz_{\ge 0}}$  by 
$t(0)=2$ and $t(n+1)=2^{t(n)}$. 
Using Knuth's up-arrow notation, 
we can represent  $t(n)=2\mathbin{\uparrow\uparrow}(n+1)$. 
We define  $k(i)=t(5i+5)$. 
We also define $\{c_{n}\}_{n\in \zz_{\ge 0}}$ by 
$c_{n}=1/t(5i+8)$ if $n=k(i)+1$ for some $i$; otherwise $c_{n}=2^{-1}$. 
Since
$k(i)\asymp t(5i+5)$, 
 $(k(i+1)-k(i)-1)\asymp t(5i+10)$, and 
$-\log(c_{0}\cdots c_{k(i)+1})\asymp t(5i+7)$, 
and since $t(n)=o(t(n+1)) \ (n\to \infty)$, 
the conditions 
(\ref{item:kimono})--(\ref{item:end}) are satisfied. 

Put $\mathbf{m}=\mathbf{2}$, i.e., $m_i=2$ for all $i\in \zz_{\ge 0}$. 
We define $\alpha\in \shrinkset$ by 
$\alpha(n)=c_{0}\cdots c_{n}$. 
The condition (4.A) implies   
\begin{align*}
\frac{\log(m_{0}\cdots m_{k(i)})}{-\log \alpha(k(i)+1)}=
\frac{(k(i)+1)\log 2}{-\log (c_{0}\cdots c_{k(i)+1})}\to 0\ 
 (\text{as $i\to \infty$}). 
\end{align*}
Thus
$\hdim(S(\mathbf{2}), \alpha_{\sharp})=0$. 
The conditions (\ref{item:eqci1/2}), (4.B) and (4.C) imply 
\begin{align*}
\frac{\log(m_{0}\cdots m_{k(i+1)})}{-\log \alpha(k(i+1))}&=
\frac{(k(i+1)+1)\log2}{-\log (c_{0}\cdots c_{k(i+1)})}\\
&=
\frac{(k(i+1)+1)\log2}{-\log (c_{0}\cdots c_{k(i)+1})+(k(i+1)-k(i)-1)\log 2}\\
&=
\frac{\frac{k(i)+2}{k(i+1)-k(i)-1}\log 2+\log 2}
{\frac{-\log (c_{0}\cdots c_{k(i)+1})}{k(i+1)-k(i)-1}+\log 2}
\to 1 \ (\text{as $i\to \infty$}).
\end{align*}
Hence
$1\le \pdim(S(\mathbf{2}), \alpha_{\sharp})$. 
By the condition (\ref{item:inqci1/2}), 
 and 
 Lemma \ref{lem:upadim2}, we have $\adim(S(\mathbf{2}), \alpha_{\sharp})\le 1$. 
Therefore the space  $(S(\mathbf{2}), \alpha_{\sharp})$ is of 
the dimensional type 
$(0, 1, 1, 1)$. 
\end{proof}

\begin{lem}\label{lem:0iii}
There exist 
$\mathbf{m}\in \numset$ and $\alpha\in \shrinkset$
such that the space  $(S(\mathbf{m}), \alpha_{\sharp})$ is 
of the dimensional type
$(0, \infty, \infty, \infty)$. 
\end{lem}
\begin{proof}
Take  $\mathbf{m}=\{m_{i}\}_{i\in \zz_{\ge 0}}\in \numset$ and a sequence  
$\{c_{i}\}_{i\in \zz_{\ge 0}}$ in $(0, 1)$
satisfying that 
 \begin{align}
 &\lim_{n\to \infty}\frac{\log(m_{0}\cdots m_{n})}{-\log (c_{0}\cdots c_{n+1})}=0,\label{align:n/n+1=0}\\
&\lim_{n\to \infty}\frac{\log(m_{0}\cdots m_{n})}{-\log (c_{0}\cdots c_{n})}=\infty.\label{align:n/n=infty} 
\end{align}

For example, 
if 
we put $g(n)=2^{n!}$, and 
we define $m_{i}=2^{g(2i+2)}$ and 
$c_{i}=2^{-g(2i+1)}$, 
then the conditions 
(\ref{align:n/n+1=0}) and (\ref{align:n/n=infty}) are satisfied.

We define $\alpha\in \shrinkset$ by 
$\alpha(n)=c_{0}\cdots c_{n}$. 
Then, by Proposition \ref{prop:hpCantor},  
the space $(S(\mathbf{m}), \alpha_{\sharp})$ is 
of the dimensional type 
$(0, \infty, \infty, \infty)$. 
\end{proof}

We next construct a Cantor ultrametric space of 
the dimensional type $(0, 0, 1, 1)$. 
As noted in Remark \ref{rmk:equal}, 
the packing dimension and the upper box dimension of 
the  $(\mathbf{m}, \alpha)$-Cantor ultrametric space 
coincide with each other. 
Thus, 
 a metric space  of 
the dimensional type $(0, 0, 1, 1)$ is
not the $(\mathbf{m}, \alpha)$-Cantor ultrametric space for any $\mathbf{m}\in \numset$
and any $\alpha\in \shrinkset$. 
We realize it as a subspace of 
the 
 $(\mathbf{m}, \alpha)$-Cantor ultrametric space using 
the following theorem  due to  Mi\v{s}\'{i}k--\v{Z}\'{a}\v{c}ik  \cite[Theorem 4]{mivsik1990some}. 
\begin{thm}\label{thm:mz}
Let $(X, d)$ be an infinite compact metric space. 
Then for every   $w\in [0, \ubdim(X, d)]$, 
there exist a convergent 
sequence $\{a_{i}\}_{i\in \zz_{\ge 0}}$ 
with its limit 
$l\in X$ 
such that  
$\ubdim(\{l\}\cup \{\, a_{i}\mid i\in \zz_{\ge 0}\, \}, d)=w$. 
\end{thm}

\begin{lem}\label{lem:0011}
There exists
a Cantor ultrametric space  
of the dimensional type $(0, 0, 1, 1)$. 
\end{lem}
\begin{proof}
We define $\alpha\in \shrinkset$ by 
$\alpha(i)=2^{-i}$. From  Lemma \ref{lem:1111}, 
it follows that the space $(S(\mathbf{2}), \alpha_{\sharp})$  is of the dimensional type $(1, 1, 1, 1)$. 
Theorem \ref{thm:mz} implies that there exists a convergent  sequence $\{a_{i}\}_{i\in \zz_{\ge 0}}$ in $(S(\mathbf{2}), \alpha_{\sharp})$
with its limit  $l\in S(\mathbf{2})$ such that 
$\ubdim(\{l\}\cup\{\, a_{i}\mid i\in\zz_{\ge 0}\, \}, \alpha_{\sharp})
=\ubdim(S(\mathbf{2}), \alpha_{\sharp})=1$. 
For each $n\in \zz_{\ge 0}$, 
let  $r_{n}$ be the distance between  $\{a_{n}\}$ and 
$\{l\}\cup\{\, a_{i}\mid i\in\zz_{\ge 0}, i\neq n\, \}$. 
Note that $r_{n}>0$ and $r_{n}\to 0$ as $n\to \infty$. 
For each $i\in \zz_{\ge 0}$, 
let $K_{i}$ be a Cantor subspace of $B(a_{i}, r_{i}/3)$ of 
the dimensional type $(0, 0, 0, 0)$ (see Lemmas 
\ref{lem:shiftedCantor} and \ref{lem:0000}). 
We define 
$K=\{l\}\cup\bigcup_{i\in \zz_{\ge 0}}K_{i}$. 
Then,
the space 
$K$ is a Cantor space
(see \cite[Theorem 30.3]{W1970}). 
Using  the statement (\ref{item:countablestab}) in Proposition \ref{prop:propdim}, we have 
$\hdim(K, \alpha_{\sharp})=\pdim(K, \alpha_{\sharp})=0$.  
Due to 
the statement (\ref{item:monotone}) in the Proposition \ref{prop:propdim},  we find that 
$\ubdim(K, \alpha_{\sharp})=\adim(K, \alpha_{\sharp})=1$. 
Therefore the space
$(K, \alpha_{\sharp})$ is a Cantor ultrametric space  of the dimensional type 
$(0, 0, 1, 1)$. 
\end{proof}

\begin{lem}\label{lem:00ii}
There exists
a Cantor ultrametric space  
of the dimensional type  $(0, 0, \infty, \infty)$. 
\end{lem}
\begin{proof}
We define $\beta\in \shrinkset$ by 
$\beta(i)=(i+1)^{-1}$.  
Then $(S(\mathbf{2}), \beta_{\sharp})$ is of the dimensional type $(\infty, \infty, \infty, \infty)$ (see Lemma \ref{lem:iiii}). 
By replacing the role of $(S(\mathbf{2}), \alpha_{\sharp})$ with 
that of 
$(S(\mathbf{2}), \beta_{\sharp})$  in the proof of 
Lemme \ref{lem:0011}, 
we obtain Lemma \ref{lem:00ii}. 
\end{proof}
\begin{lem}\label{lem:0001}
There exist 
$\mathbf{m}\in \numset$ and $\alpha\in \shrinkset$ 
such that the space  $(S(\mathbf{m}), \alpha_{\sharp})$ is 
of the dimensional type  $(0, 0, 0, 1)$. 
\end{lem}
\begin{proof}
The construction in this proof  has already appeared in 
\cite[Proposition 8.7]{Ishiki2019}. 
We define  $\alpha\in \shrinkset$ by 
$\alpha (n)=2^{-n^{3}}$. 
Let $\theta\in \shrinkset$ be the  shrinking sequence such that  
\[
\{\, \theta(n)\mid n\in \zz_{\ge 0}\, \}=
\alpha(\zz_{\ge 0})\cup \{\,2^{-k}\alpha(n)\mid n\in \zz_{\ge 0}, k=1,\dots, n\,\}. 
\] 
Let  $\varphi \colon \zz_{\ge 0}\to \zz_{\ge 0}$ be a map 
satisfying  that 
 $\theta(\varphi(n))=\alpha(n)=2^{-n^3}$ for all 
 $n\in \zz_{\ge 0}$. 
Let  $\psi\colon \zz_{\ge 0}\to \zz_{\ge 0}$ be a map satisfying  that 
$\varphi(\psi(n))\le n< \varphi(\psi(n)+1)$ for all 
$n\in \zz_{\ge 0}$. 
 Then, $\varphi(n)\asymp n^{2}$, and hence 
 $\psi(n)\asymp \sqrt{n}$. 
Since 
$\alpha(\psi(n))\le \theta(n)< \alpha(\psi(n)+1)$ for all $n\in \zz_{\ge 0}$,  
we have $-\log \theta(n) \asymp n^{3/2}$. 
Using  Lemma \ref{lem:numCantor}, 
we find
$\mnum_{\theta_{\sharp}}(S(\mathbf{2}), \theta(n))
=2^{n}$. 
Thus $\ubdim(S(\mathbf{2}), \theta_{\sharp})=0$. 
Next we estimate the Assouad dimension. 
Due to  Lemma \ref{lem:upadim2}, we obtain 
$\adim(S(\mathbf{2}), \theta_{\sharp})\le 1$. 
By the definition of $\theta$, 
for all $x\in S(\mathbf{2})$ and $n\in \zz_{\ge 0}$, 
we have 
$\mnum_{\theta_{\sharp}}
(B(x, 2^{-n^{3}}), 2^{-n}\cdot 2^{-n^{3}})
=2^{n+1}$, 
thus
$1\le \limsup_{\epsilon \to 0}\eta_{S(\mathbf{2}), \theta_{\sharp}}(\epsilon)$. 
Hence $\adim(S(\mathbf{2}), \theta_{\sharp})=1$. 
\end{proof}

\begin{lem}\label{lem:000i}
There exists a Cantor ultrametric space 
of $(0, 0, 0, \infty)$. 
\end{lem}
\begin{proof}
We define  $\alpha\in \shrinkset$ by 
$\alpha (n)=2^{-n^{3}}$. 
Let $\theta\in \shrinkset$ be the  shrinking sequence such that  
\[
\{\, \theta(n)\mid n\in \zz_{\ge 0}\, \}=
\alpha(\zz_{\ge 0})\cup 
\left\{\,\frac{k}{k+1}\alpha(n)
\ \middle | \ 
n\in \zz_{\ge 0}, k=1,\dots, n\,\right\}. 
\]
Using  the same method 
 as the proof of Lemma \ref{lem:0001}, 
 we conclude that  
$\ubdim(S(\mathbf{2}), \theta_{\sharp})=0$. 
Next we estimate the Assouad dimension. 
By the definition of $\theta$, 
for all $x\in S(\mathbf{2})$ and $n\in \zz_{\ge 0}$, 
we have 
\[
\mnum_{\theta_{\sharp}}
(B(x, 2^{-n^{3}}), 2^{-1}\cdot 2^{-n^{3}})
=2^{n+1}. 
\] 
Thus $\Theta_{S(\mathbf{2}), \theta_{\sharp}}(1/2)=\infty$. 
According to   Corollary  \ref{cor:infiniteadimcut},  we conclude that  
  $\adim(S(\mathbf{2}), \theta_{\sharp})=\infty$. 
\end{proof}

By the definition of ultrametrics and 
by 
Proposition \ref{prop:propdim}, we obtain:
\begin{lem}\label{lem:1234}
Let $(a_{1}, a_{2}, a_{3}, a_{4})\in \linset$, and 
let $(X, d)$ be an ultrametric space of the dimensional type $(a_{1}, a_{2}, a_{3}, a_{4})$. 
Then for every $\eta\in (0, \infty)$, 
the function $d^{1/\eta}$ defined by 
$d^{1/\eta}(x, y)=(d(x, y))^{1/\eta}$ 
belongs to 
$\ult(X)$, and 
the space $\left(X, d^{1/\eta}\right)$ is of the dimensional type 
$(\eta a_{1}, \eta a_{2}, \eta a_{3}, \eta a_{4})$. 
\end{lem}

We now prove Theorems 
\ref{thm:prescribed} and 
\ref{thm:prescribed2}. 
\begin{proof}[Proof of Theorem \ref{thm:prescribed}]
Let $(a_{1}, a_{2}, a_{3}, a_{4})\in \linset$. 
First we handle  the case of  $a_{i}<\infty$ for all $i\in \widehat{4}$. 
If $a_{i}=0$ for all $i\in \widehat{4}$, 
then 
Lemma \ref{lem:0000} implies the theorem. 
We may assume that $a_{i}>0$ for some $i$. 
Let 
$(X_{1}, d_{1})$, $(X_{2}, d_{2})$, $(X_{3}, d_{3})$, 
and 
$(X_{4}, d_{4})$
be Cantor ultrametric spaces of 
of the dimensional type $(1, 1, 1, 1)$, 
 $(0, 1, 1, 1)$,  $(0, 0, 1, 1)$, and $(0, 0, 0, 1)$, 
 respectively     
(see Lemmas \ref{lem:1111}, \ref{lem:0111}, 
\ref{lem:0011}, and \ref{lem:0001}). 
Let $P$ be the set of all $i\in \widehat{4}$ with $a_{i}>0$. 
Remark that $P\neq \emptyset$. 
Put $Y=\coprod_{i\in P}X_{i}$. 
Take $r\in \met(P)$. 
Applying  Lemma \ref{lem:ultraamal} to $r$ and 
$\{d_{i}^{1/a_{i}}\}_{i\in P}$, we obtain an 
ultrametric 
$e$ on $Y$ such that $e|_{X_{i}^{2}}=d_{i}^{1/a_{i}}$ for all 
$i\in P$. 
Notice that  $(Y, e)$ is a Cantor  space. 
According to the statement (\ref{item:finitestab}) in 
Proposition \ref{prop:propdim},  
and  Lemma \ref{lem:1234}, 
we conclude that 
$(Y, e)$ is of the dimensional type $(a_{1}, a_{2}, a_{3}, a_{4})$. 
By a similar method  using Lemmas \ref{lem:iiii}, \ref{lem:0iii}, \ref{lem:00ii}, 
and 
\ref{lem:000i},  we can prove  the theorem  in other cases. 
This finishes the proof of Theorem \ref{thm:prescribed}. 
\end{proof}

\begin{proof}[Proof of Theorem \ref{thm:prescribed2}]
Let $(l, a_{1}, a_{2}, a_{3}, a_{4})\in \linsettwo$. 
Based on   Theorem \ref{thm:prescribed}, 
we may assume that $l>0$. 
In the case of  $l<\infty$, 
let $(M, e)$ be the metric space $([0, 1]^{l}, d_{\rr^{l}})$, where $d_{\rr^{l}}$ is the Euclidean metric. 
In the case of  $l=\infty$, 
 let $(M, e)$ be the metric space 
$(\qcube, u)$, where 
$u\in \met(\qcube)$. 
In any case, we notice that  the space $(M, e)$ is of the dimensional type 
$(l, l, l, l)$ and $\tdim M=l$. 
Due to  Theorem \ref{thm:prescribed}, 
we can find 
 a Cantor ultrametric space $(X, d)$ of the dimensional type 
$(a_{1}, a_{2}, a_{3}, a_{4})$. 
Put $Y=X\sqcup M$. 
Let $h$ be a metric in $\met(Y)$ with 
$h|_{X^{2}}=d$ and $h|_{([0, 1]^{l})^{2}}=e$ 
(see Lemma \ref{lem:amal2}). 
Since $\tdim M=l$, 
by Proposition \ref{prop:tdimprod}, 
we conclude that $(Y, h)$ is a metric space as desired. 
This completes  the proof of Theorem \ref{thm:prescribed2}. 
\end{proof}

\section{Topological embeddings of the Hilbert cube}\label{sec:embed}

\subsection{Construction of embeddings}
We refer to the construction in \cite{Ishiki2021branching}.

\begin{df}\label{df:tail}

We define 
$\ell: \rr\to [0, \infty)$ 
by 
$\ell(x)=\sqrt{2-2\cos(x)}$, 
and 
define a map 
$\theta\colon [0, 1]\to [\pi/6, \pi/3]$ 
by 
$\theta(t)=(\pi/6)(t+1)$. 
For $q\in \qcube$, and $j\in \zz_{\ge 0}$
we define 
a metric 
$e_{j, q}\in \met(\widehat{3})$ 
by 
\[
e_{j, q}(a, b)=
\begin{cases}
2^{-j-1} & \text{if  
$\{a, b\}=\{1, 2\}$ or $ \{2, 3\}$;}\\
2^{-j-1}\cdot \ell(\theta(q_{j})) & 
\text{if $\{a, b\}=\{1, 3\}$. }
\end{cases}
\]
The metric space 
$\left(\widehat{3}, e_{j, q}\right)$
is the set of vertices of the isosceles triangle whose
apex angle is 
$\theta(q_{j})$ 
and 
whose length of the legs is 
$2^{-j-1}$. 

Next we define 
$\ooo=
\left(\widehat{3}\times \zz_{\ge 0}\right)\sqcup
\{\infty\}$. 
To simplify our description, 
for $o\in \widehat{3}$ and $i\in \zz_{\ge 0}\sqcup\{\infty\}$, 
we represent an element of $\ooo$ as 
$o_{i}=(o, i)\in \widehat{3}\times \zz_{\ge 0}$ if $i\neq \infty$; otherwise,  $1_{\infty}=2_{\infty}=3_{\infty}=\infty$.  
For $q\in \qcube$, 
we  define a metric $u[q]$ on $\ooo$ by 
\[
u[q](o_{i}, o'_{j})=
	\begin{cases}
		e_{i, q}(o, o') & \text{if $i=j$,}\\
		|2^{-i} -2^{-j}| & \text{if $i\neq j$, }\\
		2^{-i} & \text{if $o_{i}=\infty$ and $o'_{j}\neq \infty$, }
	\end{cases}
\]
and we also define an ultrametric $v[q]$ on $\ooo$ by 
\[
v[q](o_{i}, o'_{j})=
	\begin{cases}
		e_{i, q}(o, o') & \text{if $i=j$,}\\
		\max\{2^{-i}, 2^{-j}\} & \text{if $i\neq j$, }\\
		2^{-i} & \text{if $o_{i}=\infty$ and $o'_{j}\neq \infty$. }
	\end{cases}
\]
These constructions have already appeared in 
\cite{Ishiki2019} as \emph{telescope spaces}. 
Note that 
$u[q]$ and $v[q]$ generate the same topology, 
which   makes $\ooo$
a countable compact metrizable space 
possessing the unique accumulation 
point $\infty$. 
\end{df}

\begin{lem}\label{lem:ooo0000}
For all $q\in \qcube$, 
the spaces $(\ooo, u[q])$ and $(\ooo, v[q])$ are of the dimensional type 
$(0, 0, 0, 0)$. 
\end{lem}
\begin{proof}
For each $o\in \widehat{3}$, 
we define 
$A_{o}=\{\, o_{i}\mid i\in \zz_{\ge 0}\cup \{\infty\, \}\}$. 
Note that $\Omega=A_{1}\cup A_{2}\cup A_{3}$. 
By the definition of $u[q]$, each  $A_{o}$ is isometric  to 
the space $\{0\}\cup\{2^{-i}\mid i\in \zz_{\ge 0}\}$ equipped with the Euclidean metric.
According to 
\cite[Theorem 14]{larman1967new} or 
\cite[Corollary 7]{garcia2018assouad}, 
the Assouad dimension of 
$\{0\}\cup\{2^{-i}\mid i\in \zz_{\ge 0}\}$ is $0$. 
Using the statement (\ref{item:finitestab}) in Proposition \ref{prop:propdim}, 
we find that $\adim(\ooo, u[q])=0$.
Since $(\ooo, v[q])$ is  bi-Lipschitz 
equivalent to $(\ooo, u[q])$ (see \cite[Proposition 3.6]{Ishiki2019}), 
by (\ref{item:biholder}) in  Proposition  \ref{prop:propdim}, 
we have $\adim(\ooo, v[q])=0$.
\end{proof}

The proofs of the metric parts of  the following two propositions  are 
presented in \cite[Propositions 4.5 and 4.6]{Ishiki2021branching}. 
The ultrametric parts can be proven by the same method. 

\begin{prop}\label{prop:contiq}
The maps  $U\colon \qcube\to \pmet(\ooo)$ 
and $V\colon \qcube\to \pumet(\ooo)$
defined by 
$U(q)=u[q]$ and $V(q)=v[q]$ are
continuous. 
\end{prop}

\begin{prop}\label{prop:Uiso}
Let 
$q, r\in \qcube$, 
and  
$K, L\in (0, \infty)$. 
Assume that  the spaces $(\ooo, K\cdot u[q])$ 
and 
$(\ooo, L\cdot u[r])$,  
or the spaces $(\ooo, K\cdot v[q])$ and   
$(\ooo, L\cdot v[r])$ 
are isometric to each other.   
Then we have $q=r$. 
\end{prop}

The following proposition occupies 
the main part of the proof of 
Theorem \ref{thm:topemb}. 
\begin{prop}\label{prop:embednm}
Let $(l, a_{1}, a_{2}, a_{3}, a_{4})\in \linsettwo$, 
 $n\in \zz_{\ge 1}$ and $m\in \zz_{\ge 2}$. 
Let $H$ be a compact metrizable space and 
$\{v_{i}\}_{i=1}^{n+1}$ be $n+1$ points in $H$. 
Put $\deltasecond=H\setminus \{\, v_{i}\mid i=1, \dots, n+1\, \}$. 
Let $\{(X_{i}, d_{i})\}_{i=1}^{n+1}$ be a 
sequence of compact metric spaces satisfying  that 
$\grdis((X_{i}, d_{i}), (X_{j}, d_{j}))>0$ for all distinct $i, j$. 
Then there exists a continuous map 
$F\colon H\times \widehat{m}\to \dimset(l, a_{1}, a_{2}, a_{3}, a_{4})$
such that 
\begin{enumerate}
\item 
for all $i\in \widehat{n+1}$ and $j\in \widehat{m}$ we have 
$F(v_{i}, j)=(X_{i}, d_{i})$;
\item 
for all $(s, i), (t, j)\in \deltasecond\times \widehat{m}$ with $(s, i)\neq (t, j)$, we have 
$F(s, i)\neq F(t, j)$. 
\end{enumerate}
\end{prop}
\begin{proof}
Since the space 
$H\times \widehat{m}$ is compact and metrizable, 
there exists a topological embedding 
$\tau\colon H\times \widehat{m}\to \qcube$ 
(this is the Urysohn Metrization  Theorem, see \cite{Kelly1975}). 
Since every metrizable space is perfectly normal 
(see \cite[Proposition 4.18]{P1975}), 
 for each $i\in \widehat{n+1}$, there exists a continuous function 
$\zeta_{i}\colon H\to [0, 1]$ such that 
$\zeta_{i}^{-1}(0)=\{\, v_{k}\mid k\neq i \, \}$ and 
$\zeta_{i}^{-1}(1)=\{v_{i}\}$, and there exists 
 a continuous function 
 $\xi\colon H\to [0, 1]$
 with 
$\xi^{-1}(0)=\{\, v_{i}\mid i=1, \dots, n+1\, \}$. 
Due to Lemma \ref{lem:0000}, 
we can take   a Cantor ultrametric space
$(P, v)$ of the 
dimensional type $(0, 0, 0, 0)$. 
For each  $s\in H$, and for each 
$i\in \widehat{n+1}$,  
we define $(Y_{i}, h_{i, s})=
(X_{i}\times P, d_{i}\times_{\infty} (\xi(s)\cdot  v))$, and 
we define $Y_{n+2}=\ooo$.  
For all $i\in \widehat{n+2}$, $s\in H$, 
and $k\in \widehat{m}$, 
we also define  a metric $E_{i, s, k}$ on $Y_{i}$ by 
\[
E_{i, s, k}(x, y)=
\begin{cases}
\zeta_{i}(s)\cdot h_{i, s}(x, y) & \text{if $i\neq n+2$;}\\
\xi(s)\cdot u[\tau(s, k)](x, y)& \text{if $i=n+2$}. 
\end{cases}
\]

Put 
$Z=\coprod_{i=1}^{n+2}Y_{i}$,  and 
take $r\in \met(\widehat{n+2})$ and $p_{i}\in Y_{i}$. 
For each $(s, k)\in H\times \widehat{m}$,   we define a symmetric function  $D_{s, k}$ on $Z$ by 
\begin{align*}
&D_{s, k}(x, y)=\\
&\begin{cases}
E_{i, s, k}(x, y) & \text{if $x, y\in Y_{i}$;}\\
E_{i, s, k}(x, p_{i})+\xi(s)r(i, j)+E_{j, s, k}(p_{j}, y) 
& \text{if $x\in Y_{i}$ and $y\in Y_{j}$. }
\end{cases}
\end{align*}
Lemma \ref{lem:amal2} implies that 
 $D_{s, k}$ is a pseudo-metric on $Z$ for all 
 $(s, k)\in H\times \widehat{m}$. 
We notice 
 that
$D_{s, k}$ is  a metric if and only if 
 $s\in \deltasecond$. 
 We also  notice that 
 for all 
 $i\in \widehat{n+1}$ and  $k\in \widehat{m}$,   
  the quotient metric  space 
$\left(Z_{/D_{v_{i}, k}}, [D_{v_{i}, k}]\right)$ is isometric to 
$(X_{i}, d_{i})$. 
Combining  Propositions  \ref{prop:propdim}, 
\ref{prop:tdimprod}, and \ref{prop:productdim},  and  
Lemma \ref{lem:ooo0000},  we conclude that 
$(Z, D_{s, k})\in \dimset(l, a_{1}, a_{2}, a_{4}, a_{4})$ for all $(s, k)\in \deltasecond\times \widehat{m}$. 

We define a map 
$F\colon H\times \widehat{m}\to \dimset(l, a_{1}, a_{2}, a_{3}, a_{4})$  by 
\[
F(s, k)=
\begin{cases}
(X_{i}, d_{i}) & \text{if $s=v_{i}$ for some $i\in \widehat{n+1}$;}\\
\left(Z, D_{s, k}\right) & \text{otherwise.}
\end{cases}
\]
Then the condition (1) is satisfied. 
By the definition of $D_{s, k}$, and by 
Proposition \ref{prop:contiq}, 
the map 
$\yoavoidmap\colon H\times \widehat{m}\to \pmet(Z)$ defined by $W(s, k)=D_{s, k}$ is continuous. 
Therefore, 
 Corollary \ref{cor:metconti} 
guarantees the continuity of 
the map $F$. 

Next we prove the condition (2). 
For a metric space $(S, h)$, 
we denote by  $\mathcal{CI}(S, h)$ the closure of 
the set of all isolated point of $(S, h)$. 
Note that if metric spaces $(S, h)$ and $(S', h')$ are isometric to each other, then so are 
$\mathcal{CI}(S, h)$ and $\mathcal{CI}(S', h')$. 
Since 
$(P, v)$ has no isolated points, so does $Y_{i}$
for all $i\in \widehat{n+1}$. 
Then, by the definitions of $\ooo$, and  $D_{s, k}$, 
the space $\mathcal{CI}(Z, D_{s, k})$ is 
isometric to 
$(\ooo, \xi(s)\cdot u[\tau(s, k)])$ for 
all $(s, k)\in \deltasecond\times \widehat{m}$. 
Since $\mathcal{CI}$ is isometrically invariant, 
and since $\tau$ is injective, 
Proposition \ref{prop:Uiso} implies that  
the condition (2) is satisfied. 
This finishes the proof. 
\end{proof}

Similarly to 
Proposition 
\ref{prop:embednm}, 
we obtain its analogue for ultrametrics. 
\begin{prop}\label{prop:embednmult}
Fix $n\in \zz_{\ge 0}$. 
Let $H$ be a compact metrizable space and 
$\{v_{i}\}_{i=1}^{n+1}$ be $n+1$ points in $H$. 
Put $\deltasecond=H\setminus \{\, v_{i}\mid i=1, \dots, n+1\, \}$. 
Let
$m\in \zz_{\ge 2}$. 
Let $\{(X_{i}, d_{i})\}_{i=1}^{n+1}$ be a 
sequence of compact ultrametric spaces with 
$\grdis((X_{i}, d_{i}), (X_{j}, d_{j}))>0$ for all distinct $i, j$. 
Then there exists a continuous map 
$F\colon H\times \widehat{m}\to \ultsp$
such that 
\begin{enumerate}
\item 
for all $i\in \widehat{n+1}$ and $j\in \widehat{m}$ we have 
$F(v_{i}, j)=(X_{i}, d_{i})$;
\item 
for all $(s, i), (t, j)\in \deltasecond\times \widehat{m}$ with $(s, i)\neq (t, j)$, we have 
$F(s, i)\neq F(t, j)$. 
\end{enumerate}

\end{prop}
\begin{proof}
By replacing the metric $u[\tau(s, k)]$ with 
the ultrametric $v[\tau(s, k)]$,  
and 
 replacing the symbol  ``$+$'' with the symbol  ``$\lor$'' in 
the definition of $D_{s, k}$ in  
the proof of Proposition \ref{prop:embednm}, 
and using Lemma \ref{lem:ultraamal} instead of Lemma 
\ref{lem:amal2}, 
we obtain the  proof of  Proposition \ref{prop:embednmult}. 
\end{proof}

We give the proof of  Theorem \ref{thm:topemb}. 
\begin{proof}[Proof of Theorem \ref{thm:topemb}]
Let $n\in \zz_{\ge 1}$, 
and $H$ a compact metrizable space. 
Take  mutually distinct 
$n+1$ points $\{v_{i}\}_{i=1}^{n+1}$ in 
$H$, and 
let $\{(X_{i}, d_{i})\}_{i=1}^{n+1}$ be a sequence in $\mathscr{S}$ such that 
$\grdis((X_{i}, d_{i}), (X_{j}, d_{j}))>0$ for all distinct $i, j$.
 
First we  assume that $\mathscr{S}=
\dimset(l, a_{1}, a_{2}, a_{3}, a_{4})$
for some  
$(l, a_{1}, a_{2}, a_{3}, a_{4})$ in $\linsettwo$. 
Put $m=n+2$. 
Let 
$F\colon H\times \widehat{m}\to \dimset(a_{1}, a_{2}, a_{3}, a_{4})$ be a map stated in 
Proposition \ref{prop:embednm}. 
For all  
$i\in \widehat{n+1}$ and 
$j\in \widehat{m}$, 
we define 
$S(i, j)=
\left\{\, s\in \deltasecond
\  \middle | \ 
F(s, j)=(X_{i}, d_{i})
\, \right\}$. 
According to  the conditions (1) and (2) in Proposition \ref{prop:embednm}, 
for all $i\in \widehat{n+1}$
the set $\bigcup_{j=1}^{m}S(i, j)$
is empty or a singleton. 
Thus,   by $m=n+2$, we obtain 
$\widehat{m}\setminus \bigcup_{i=1}^{n+1}\{\, j\in \widehat{m}\mid \card(S(i, j))>0\, \}\neq \emptyset$, 
and we can take $k$ from this set. 
Then 
$\{\, (X_{i}, d_{i})\mid i=1, \dots, n+1\, \}\cap 
F\left(\deltasecond\times \{k\}\right)
=\emptyset$. 
Therefore,  the function 
$\Phi\colon H\to \dimset(a_{1}, a_{2}, a_{3}, a_{4})$ defined by $\Phi(s)=F(s, k)$ is injective, and hence $\Phi$ is  a  topological embedding since 
$H$ is compact. 
This completes the proof for 
$\mathscr{S}=\dimset(l, a_{1}, a_{2}, a_{3}, a_{4})$. 

Theorem \ref{thm:topemb} for $\mathscr{S}=\ultsp$ is 
obtained by the same method using Proposition 
\ref{prop:embednmult} instead of Proposition
 \ref{prop:embednm}. 
 This finishes the proof of Theorem \ref{thm:topemb}. 
\end{proof}

We next show Corollaries 
\ref{cor:dimset} and 
\ref{cor:pathinfult}. 
\begin{proof}[Proof of Corollary \ref{cor:dimset}]
Applying Theorem \ref{thm:topemb} 
to $H=[0, 1]$ and $H=\qcube$, 
it is shown that 
$\dimset(l, a_{1}, a_{2}, a_{3}, a_{4})$ is path-connected and infinite-dimensional, respectively. 
By the same method as \cite[Lemmas 2.17 and  2.18]{Ishiki2021branching}
using the property (\ref{item:finitestab}) in Proposition 
\ref{prop:propdim}, 
we conclude that $\dimset(l, a_{1}, a_{2}, a_{3}, a_{4})$ is dense in $\grsp$. 
\end{proof}

\begin{proof}[Proof of Corollary \ref{cor:pathinfult}]
Similarly to the proof of Corollary \ref{cor:dimset}, 
applying Theorem \ref{thm:topemb} 
to $H=[0, 1]$ and $H=\qcube$, 
it is shown that 
$\ultsp$ is path-connected and infinite-dimensional, respectively. 
\end{proof}
\begin{rmk}
M{\'e}moli, Smith and Wan \cite{memoli2019gromov}
introduced the \emph{$p$-metrics} as follows: 
For $p\in [1, \infty]$,  and  $a, b\in  [0, \infty)$, 
we define 
$a \pppp b= (a^{p}+b^{p})^{1/p}$ if 
$p\neq \infty$; 
otherwise $\max\{a, b\}$. 
A metric $d$  on a set $X$ is a \emph{$p$-metric} 
if 
$d(x, y)\le d(x, z)\pppp d(z, y)$ for all $x, y, z\in X$. 
We denote by $\grsp_{p}$ the set of  all compact $p$-metric spaces in $\grsp$. 
By replacing the symbol ``$+$'' with ``$\pppp$'', 
we can prove $p$-metric analogues of 
Lemma \ref{lem:amal2}, 
Proposition \ref{prop:embednm}, 
and 
Theorem \ref{thm:topemb}. 
From these arguments,  we can  deduce the    
 path-connectedness and the infinite-dimensionality 
 of $(\grsp_{p}, \grdis)$ for all $p\in [1, \infty]$
 (compare with \cite[Theorem 7.11]{memoli2019gromov}).  
Since $1$-metrics and $\infty$-metrics are identical with 
ordinal metrics and ultrametrics, respectively, 
the cases of $p=1, \infty$ are contained  in  Theorem \ref{thm:topemb}.  
\end{rmk}

\subsection{The non-Archimedean Gromov--Hausdorff space}\label{subsec:ultra}
Let 
$(X, d)$ 
and 
$(Y, e)$ be 
ultrametric spaces. 
We define the 
\emph{non-Archimedean Gromov--Hausdorff distance} 
$\ugrdis((X, d),(Y, e))$ between 
$(X, d)$ 
and 
$(Y, e)$ 
 as 
the infimum of  all
$\hdis(i(X), j(Y); Z, h)$, 
where 
$(Z, h)$ 
is an ultrametric space, 
 and 
$i\colon X\to Z$ 
and 
$j\colon Y\to Z$ 
are isometric embeddings. 

The proof of the following 
proposition is presented in \cite{Z2005}. 
\begin{prop}\label{prop:nault}
The non-Archimedean Gromov--Hausdorff distance 
$\ugrdis$ is an ultrametric on $\ultsp$. 
\end{prop}

We shall prove 
Corollaries \ref{cor:discon}
and \ref{cor:ratio}. 
\begin{proof}[Proof of Corollary  \ref{cor:discon}]
By Corollary \ref{cor:pathinfult}, 
the space $(\ultsp, \grdis)$ is path-connected; however 
$(\ultsp, \ugrdis)$ is not so  since $\ugrdis$ is an ultrametric (Proposition \ref{prop:nault}). Thus $I_{\ultsp}$ is not continuous. 
Note that we can use the infinite-dimensionality instead of 
the path-connectedness of $(\ultsp, \grdis)$ since 
all ultrametric spaces have zero topological dimension.
\end{proof}

\begin{proof}[Proof of Corollary  \ref{cor:ratio}]
For the sake of contradiction, 
suppose that  there exists $c\in [2, \infty)$
such that for all $(X, d), (Y, e)\in \ugrdis$
we have 
 \[
 \ugrdis((X, d), (Y, e))\le c\cdot \grdis((X, d), (Y, e)).
 \] 
 Then $I_{\ultsp}$ is continuous. 
 This contradicts Corollary \ref{cor:discon}. 
\end{proof}

Moreover, we can obtain 
 stronger versions of 
Corollaries  \ref{cor:discon} and \ref{cor:ratio}. 
Before proving these statements, we prepare two propositions 
on
$\ugrdis$. 
The following is \cite[(3B) in Theorem 4.2]{qiu2009geometry}. 
\begin{prop}\label{prop:ugrdis}
Let $(X, d)$ and $(Y, e)$ be 
bounded ultrametric spaces. 
If $\delta_{d}(X)\neq \delta_{e}(Y)$, 
then $\ugrdis((X, d), (Y, e))=\max\{\delta_{d}(X), \delta_{e}(Y)\}$. 
\end{prop}

By Proposition \ref{prop:ugrdis}, we obtain:
\begin{prop}\label{prop:epsilonmod}
Let $(X, d)$ be a compact ultrametric space, and 
$\epsilon\in (0, \infty)$. 
Put $d_{\epsilon}=(1+\epsilon)d$. 
Then $d_{\epsilon}$ is an ultrametric and we have 
\begin{align}
&\grdis((X, d), (X, d_{\epsilon}))\le \epsilon\delta_{d}(X);\label{eq:ep}\\
&\ugrdis((X, d), (X, d_{\epsilon}))
= (1+\epsilon)\delta_{d}(X).\label{eq:disdis}
\end{align}
\end{prop}

Proposition \ref{prop:epsilonmod}
implies   stronger versions of 
Corollaries  \ref{cor:discon} and \ref{cor:ratio}. 
\begin{cor}\label{cor:discon2}
The map  $I_{\ultsp}\colon (\ultsp, \grdis)\to (\ultsp, \ugrdis)$ 
 is not continuous at any point in $\grsp$
 except the 
one-point metric space. 
\end{cor}

\begin{cor}\label{cor:ratio2}
For every $c\in [2, \infty)$, and for every 
compact ultrametric space $(X, d)$ containing at least 
two points, 
there exists a  compact ultrametric space 
$(Y, e)$ such that 
\[
c\cdot \grdis((X, d), (Y, e))<\ugrdis((X, d), (Y, e)). 
\]
Moreover, $(Y, e)$ can be chosen as 
a metric space homeomorphic to $(X, d)$. 
\end{cor}

\begin{rmk}
Corollaries 
\ref{cor:pathinfult}
\ref{cor:discon}, 
 \ref{cor:discon2},  and 
 \ref{cor:ratio2} 
can be considered as partial answers of Question 2.18 in \cite{qiu2009geometry}  asking 
the relations between 
$(\ultsp, \grdis)$ and $(\ultsp, \ugrdis)$. 
\end{rmk}

\bibliographystyle{amsplain}
\bibliography{bibtex/emb.bib}

\end{document}